\documentclass [12pt] {article}
\usepackage[dvipdfm]{graphicx}
\usepackage{amsfonts}
\usepackage{amsmath}
\usepackage{amssymb}
\usepackage{bm}
\usepackage[abbrev,citation-order]{amsrefs}
\usepackage{fullpage,theorem}
\usepackage{tensor}
\usepackage{upgreek}
\usepackage{slashed}

\newtheorem{thm}{Theorem}[section]
\newtheorem{cor}[thm]{Corollary}
\newtheorem{lem}[thm]{Lemma}
\newtheorem{prop}[thm]{Proposition}

\theorembodyfont{\rmfamily}
\newtheorem{defn}[thm]{Definition}

\newtheorem{rem}[thm]{Remark}

\numberwithin{equation}{section}
\newenvironment{proof}{\noindent \emph{Proof.}}{\hspace{\stretch{1}}$\Box$}

\newcommand{\parderv}[2] {\frac{\partial#1}{\partial#2}}
\newcommand{\dd} {\mathrm{d}}
\newcommand{\ii} {\mathrm{i}}

\newcommand{\tr} {\mathop{\mathrm{tr}}}
\newcommand{\Dop} {\mathop{\slashed{\mathrm{D}}}}
\newcommand{\D} {\mathcal{D}}

\newcommand{\End} {\mathop{\mathrm{End}}}
\newcommand{\Hom} {\mathop{\mathrm{Hom}}}

\newcommand{\SO} {\mathop{\mathrm{SO}}}

\newcommand{\C} {\mathbb{C}}
\newcommand{\ind} {\indices}

\newcommand{\lb} [1] {\left[ #1 \right. }

\newcommand{\rb} [1] {\left. #1 \right] }
\newcommand{\orc}[1] {\mathring{#1}}

\begin{document}

\title{Killing-Yano tensors and multi--hermitian structures}
\author{Lionel Mason \& Arman Taghavi-Chabert\\ The Mathematical
  Institute, 24-29 St Giles, Oxford OX1 3LB}

\date{}
\maketitle

\begin{abstract}
  We show that the Euclidean Kerr-NUT-(A)dS metric in $2m$ dimensions
  locally admits $2^m$ hermitian complex structures. These are derived
  from the existence of a non-degenerate closed conformal Killing-Yano
  tensor with distinct eigenvalues. More generally, a conformal
  Killing-Yano tensor, provided its exterior derivative satisfies a
  certain condition, algebraically determines $2^m$ almost complex
  structures that turn out to be integrable as a consequence of the
  conformal Killing-Yano equations.  In the complexification, these
  lead to $2^m$ maximal isotropic foliations of the manifold and, in
  Lorentz signature, these lead to two congruences of null geodesics.
  These are not shear-free, but satisfy a weaker condition that also
  generalizes the shear-free condition from four dimensions to
  higher dimensions. In odd dimensions, a conformal Killing-Yano
  tensor leads to similar integrable distributions in the
  complexification.  We show that the recently discovered
  5-dimensional solution of L\"{u}, Mei and Pope also admits such
  integrable distributions, although this does not quite fit into the story as
  the obvious associated two-form is not conformal Killing-Yano.  We
  give conditions on the Weyl curvature tensor imposed by the
  existence of a non-degenerate conformal Killing-Yano tensor; these
  give an appropriate generalization of the type D condition on a Weyl
  tensor from four dimensions.
\end{abstract}

\section{Introduction}
In the construction of exact solutions to the Einstein equations in
four dimensions, a prominent role is played by shear-free congruences
of null geodesics. In vacuum, these lead, via the Goldberg-Sachs
theorem, to the algebraic degeneracy of the Weyl tensor and
considerable simplification of the gravitational field equations.  The
Kerr-Newman black-hole solutions \cites{Kerr1963,Debney1969} has degenerate Weyl tensor of type D
and such solutions are particularly well endowed in the sense that
they admit two such congruences. In higher dimensions, the Kerr-Schild and Kerr-NUT-(A)dS solutions
of \cites{Myers1986,Hawking1999,Gibbons2005,Chen2008} do have preferred null congruences, but
they are not shear-free.  In \cite{Hughston1988} it was proposed that
the appropriate higher dimensional concept to extend the
$4$-dimensional results should be that of an integrable complex
distribution $\D\subset T_\C M$, $[\D,\D]\subset \D$ that is totally
null and of maximal dimension.  With this definition, a number of
$4$-dimensional results were generalized to arbitrary dimension.  In
Euclidean signature in even dimensions, this is simply a metric
compatible complex structure, i.e., a Hermitian structure.  In Lorentz
signature, $\D\cap \overline \D $ is necessarily one-dimensional and
defines a null congruence.  This is automatically shear-free in four
dimensions, but not in higher dimensions, but it Lie derives a complex
structure on the tangent space orthogonal and transverse to the null
congruence.  In this paper we show that it is these null congruences
that are relevant in the study of the higher-dimensional
Kerr-NUT-(A)dS solutions.

In four dimensions the type D condition on the Weyl curvature of a
vacuum space-time is equivalent to the condition that it admits two
distinct geodesic shear-free congruences.  It is also equivalent to
the existence of a \emph{conformal Killing-Yano tensor}, a $2$-form
$\bm{\phi}$ that, on a general $n$-dimensionial manifold, satisfies
 \begin{align*}
    \nabla _{\bm{X}} \bm{\phi} & = \frac{1}{3} \bm{X} \lrcorner
    \bm{\tau}
    + \frac{1}{n-1} \bm{X}^* \wedge \bm{K},
  \end{align*}
  for all vector fields $\bm{X}$, where $\bm{\tau}$, a $3$-form and
  $\bm{K}$, a $1$-form, are determined by the equation.  Such a form
  is said to be a {\em Killing $2$-form} if $\bm{K}\equiv0$, and a {\em
    $*$-Killing $2$-form} if $\bm{\tau} \equiv 0$.  In four dimensions
  Killing $2$-forms are mapped onto $*$-Killing 2-forms by Hodge
  duality, but in general dimension, the two concepts are distinct.
  Killing-Yano tensors and their generalisation to any $p$-forms were
  first introduced by Kentaro Yano as a natural generalisation of
  Killing vectors to forms in \cites{Yano1952, Yano1952a}. Conformal
  Killing-Yano tensors as a generalisation of conformal Killing
  vectors made their first appearance in \cites{Tachibana1969,
    Kashiwada1968}, and are often refered to as conformal Killing
  forms or twistor forms.

  Killing-Yano tensors underly much of the theory of the
  four-dimensional black hole solutions. In \cite{Carter1968}, Brandon
  Carter identifies the fourth conserved quantity in the Kerr-Newman
  black hole solution, which allows the separation of the
  Hamilton-Jacobi equations and the complete integrability of geodesic
  motions. In \cites{Walker1970, Hughston1972} it is shown that this
  `hidden' symmetry can be represented by a Killing tensor, which
  turns out to be the `square' of a conformal Killing $2$-form. In the
  same papers, a spinorial approach to the problem sheds light on the
  null geodesic shear-free congruences in the Kerr geometry: in tensor
  language, the real eigenvectors of the conformal Killing-Yano tensor
  define a pair of geodesic shear-free null congruences.

More recently, similar structures have been found for the black hole
solutions in higher dimensions.  These have been the 
object of intensive study motivated to a large extent by ideas from
string theory and M-theory.  The Kerr-NUT-(A)dS
metric is  a higher-dimensional generalisation of the Kerr
metric, generalising also the Pleba\'{n}ski-Demia\'{n}ski metric. Explicitly,
in Euclideanised form,
the $n$-dimensional Kerr-NUT-(A)dS metric is given by \cite{Chen2006}
\begin{align*}
  \bm{g} = \sum _{\mu=1}^{m} \left( \bm{e} \ind{^\mu} \odot \bm{e}
    \ind{^\mu} + \bm{e}\ind{^{m+\mu}} \odot \bm{e} 
  \ind{^{m+\mu}} \right) + \epsilon \bm{e}  \ind{^{2m+1}} \odot \bm{e}
\ind{^{2m+1}} 
\end{align*}
where, in terms of local coordinates $\left\{ x_\mu, \psi_k \right\}$,
\begin{align*}
  \bm{e} ^\mu & =  \left( \frac{U_\mu}{X_\mu}\right)^{1/2} \dd x_\mu, &
  \bm{e} ^{m + \mu} & =  \left( \frac{X_\mu}{U_\mu}\right) ^{1/2}
  \sum_{k=0} ^{m-1} A_\mu ^{(k)} \dd 
  \psi _k, &  \bm{e} ^{2m+1} & =
  \left(-\frac{c}{A^{(m)}}\right)^{1/2}\left(\sum_{k=0} ^m A^{(k)} \dd
    \psi _k \right). 
\end{align*}
with
\begin{align*}
X_\mu  = (-1)^\epsilon \frac{(g^2 x_\mu ^2 - 1)}{x_\mu ^{2 \epsilon}} \prod_{k=1} ^{m-1+\epsilon} (a_k ^2 - x_\mu ^2 )
+ 2 M_\mu (-x_\mu)^{1-\epsilon}, \qquad U_\mu  = \prod_{\substack{\nu=1\\\nu \neq \mu}} ^{m}
(x_\nu^2 -
x_\mu^2), \\
c = \prod _{k=1} ^m a_k ^2, \qquad
A_\mu ^{(k)}  =
\sum_{\substack{\nu_1 < \nu_2 < \ldots < \nu_k \\ \nu_i \neq \mu}}  x_{\nu_1}^2 x_{\nu_2}^2
\ldots x_{\nu_k}^2, \qquad A ^{(k)}  =
\sum_{\nu_1 < \nu_2 < \ldots < \nu_k}  x_{\nu_1}^2 x_{\nu_2}^2
\ldots x_{\nu_k}^2.
\end{align*}
Here, $m=[n/2 ]$, and $\epsilon = n - 2m$. The constants $a_k$,
$-\ii^{1+\epsilon}M_m$, $M_\mu$ ($\mu \neq m$) are the rotation
coefficients, the mass and the NUT parameters respectively, and
$\lambda=-g^2$ is proportional to the cosmological constant. (With
appropriate choices of the constants, Lorentzian real slices can also
be found.)  Like its
four-dimensional counterpart, the Kerr-NUT-(A)dS metric admits a closed
conformal Killing-Yano tensor
\begin{align} \label{TwKNA}
        \bm{\phi} = \sum x_\mu \bm{e}^\mu \wedge \bm{e}^{m+\mu}.
\end{align}
Aspects of the four-dimensional theory have been generalised to
Kerr-NUT-(A)dS metric in arbitrary dimensions in a series of papers
\cites{Page2007,Frolov2007,Frolov2007a,Krtouvs2007,Krtouvs2007a,
  Kubizvn'ak2007,Kubizvn'ak2007a,Sergyeyev2008, 
  Frolov2008,Hamamoto2007,Houri2007,Houri2008,Krtous2008,Oota2008}
in 
which the separation of the Hamilton-Jacobi, Klein-Gordon and Dirac
equations and the complete 
integrability of geodesic motions are dealt with. In this paper, we
turn our attention to the existence of a set of integrable almost
complex structures. Defining
\begin{align*}
  \bm{\theta} ^\mu = 2^{-1/2} (\bm{e} ^\mu + \ii \bm{e} ^{m+\mu}) \qquad \mbox{and} \qquad \bm{\bar{\theta}} ^{\bar{\mu}} = 2^{-1/2} (\bm{e} ^\mu - \ii
  \bm{e}
  ^{m+\mu})
\end{align*}
puts the Kerr-NUT-(A)dS metric into the form
\begin{align*}
  \bm{g} = \sum_\mu 2 \bm{\theta} \ind{^\mu} \odot \bm{\bar{\theta}}
  \ind{^{\bar{\mu}}} + \epsilon \bm{e} \ind{^{2m+1}} \odot \bm{e} \ind{^{2m+1}}.
\end{align*}
A straightforward computation of the Levi-Civita connection $1$-form
implies the integrability of each of the $2^m$ distributions defined as
the annihilator of a set of $m$ $1$-forms obtained by choosing one
from each pair $\left\{ \bm{\theta}^\mu, \bm{\bar{\theta}}^{\bar{\mu}}
\right\}$, $\mu=1,\ldots m$. These correspond to integrable almost
complex structures in the case $\epsilon=0$, and to integrable CR
structures in the case $\epsilon=1$.  These results are essentially
local in nature and although the complex structures will be defined on
a dense open set, they will not generally extend over the whole of the
regular space-time (thus they will not be global on $S^4$ or $S^6$).

The plan of the paper is as follows. We first recall the basic facts
concerning conformal Killing-Yano tensors and maximal isotropic
distributions while establishing the notation. We then prove our main
result on integrability, both in even and odd dimensions and discuss
the examples above in more detail.  In these examples the Killing Yano
tensor is closed.  We also study the example of the new
$5$-dimensional metrics discovered by L\"{u}, Mei and Pope \cites{Lu2009}
which we show does admit the corresponding integrable distribution,
although the most obvious choice for a conformal Killing-Yano tensor
does not seem to work.  

We go on to show how a Killing-Yano tensor imposes
algebraic restrictions on the Weyl tensor. We also study the closely
related structure of Hamiltonian $2$-forms and show that these also
lead to a family of integrable complex structures as for Killing-Yano
tensors.  In the last section, we re-express our results in terms of
spinors.  In particular, eigenspinors of the conformal Killing-Yano
tensor are shown to be pure and to determine the integrable
distributions discussed earlier. Finally, we briefly discuss further
issues arising from our discussion, the different possible reality
structure, the Kerr-Schild form of the metrics, the Kerr theorem,
degenerate Killing-Yano tensors and Killing spinors

\section{Preliminaries}
By and large we will not use the Einstein summation convention, but
will occasionally when there is no ambiguity and will warn the reader
of this. We adopt the notation that round and square brackets
enclosing a group of indices denote symmetrisation and
anti-symmetrisation respectively, e.g.
\begin{align*}
        k \ind{_{(a b)}} = \frac{1}{2!}\left( k \ind{_{a b}} + k \ind{_{b a}} \right)  \qquad \mbox{and} \qquad
        k \ind{_{[a b c]}} = \frac{1}{3!}\left( k \ind{_{a b c}} - k \ind{_{b a c}} + k \ind{_{b c a}} - k \ind{_{c b a}} + k \ind{_{c a b}} - k \ind{_{a c b}}\right).
\end{align*}
Indices are raised and lowered via the metric.
Tensorial quantities will be given in bold symbols, and scalar quantities in regular symbols.

\subsection{Conformal Killing $2$-forms}
Conformal Killing-Yano tensors are now much studied, see
\cite{Semmelmann2001} for a thorough
treatment. We shall only state results pertinent to conformal Killing
$2$-forms. In what follows, $\bm{V}^* \equiv \bm{g}(\bm{X})$ denotes the dual of a vector $\bm{X}$, and $\dd^*$ the adjoint of the exterior derivative $\dd$. On $p$-forms on an $n$-dimensional (pseudo-) Riemannian it is given by $\dd^* = (-1)^{n p+n+1} * \dd *$, where the $*$ is the Hodge duality operator.

\begin{defn}
  A \emph{conformal Killing-Yano tensor} or \emph{conformal Killing $2$-form} on an $n$-dimensional (pseudo-) Riemannian manifold $M$ is a $2$-form $\bm{\phi}$ which satisfies the following equation
  \begin{align} \label{CKYeq}
    \nabla _{\bm{X}} \bm{\phi} & = \frac{1}{3} \bm{X} \lrcorner
    \bm{\tau}
    + \frac{1}{n-1} \bm{X}^* \wedge \bm{K}, & \Biggl( \nabla \ind{_c} \phi \ind{_{a b}} \Biggr. & = \Biggl. \tau \ind{_{c a b}} + \frac{2}{n-1} g \ind{_{c \lb{a}}} K \ind{_{\rb{b}}} \Biggr)
  \end{align}
  for all vector fields $\bm{X}$. It follows at once $\bm{\tau} = \dd \bm{\phi}$ and $\bm{K} = -\dd^* \bm{\phi}$. If $\bm{\phi}$ is co-closed, i.e. $\bm{K}=0$, it is called a \emph{Killing $2$-form}.
  If $\bm{\phi}$ is closed, i.e. $\bm{\tau}=0$, it is called a \emph{$*$-Killing $2$-form}.
\end{defn}

Equation \eqref{CKYeq} is over-determined and one can show that in
the case $n \neq 4$ it is equivalent to a parallel section of the
bundle $\mathcal{E}^2(M) = \bigwedge ^2 T^* M \oplus \bigwedge ^3
T^*M \oplus \bigwedge ^1 T^* M \oplus \bigwedge ^2 T^* M$ with
respect to the Killing connection $\tilde{\nabla}$ as described in
\cite{Semmelmann2001}. An element $\bm{\Phi} = (\bm{\phi}, \bm{\tau},
\bm{K}, \bm{\chi}) \in \mathcal{E}^2(M)$ satisfies
$\tilde{\nabla} \bm{\Phi} = 0$ if and only if $\bm{\tau} = \dd
\bm{\phi}$, $\bm{K} = -\dd^* \bm{\phi}$, and $\bm{\chi} = \Delta
\bm{\phi}$ where $\Delta = \dd \dd^* + \dd^* \dd$ is the Beltrami-Laplacian on forms. The case $n=4$ necessitates a slight
modification of the prolongation in which Hodge duality must be
taken into account.
In flat space with flat coordinates $\left\{ x^a \right\}$,
integration leads to 
\begin{align} \label{flatsol}
  \bm{\phi} & = \left( \frac{1}{2}
\| \bm{x} \|^2 \orc{\bm{\chi}} - \bm{x}^* \wedge \bm{x}
\lrcorner \orc{\bm{\chi}} \right) + \bm{x}^* \wedge
{\orc{\bm{K}}} + \bm{x} \lrcorner {\orc{\bm{\tau}}} +
{\orc{\bm{\phi}}},
\end{align}
where $\bm{x} = \left(x^1, x^2, \ldots, x^n \right)$ is the position
vector field, and $ \orc{\bm{\chi}}$, ${\orc{\bm{K}}}$ ,
${\orc{\bm{\tau}}}$ and ${\orc{\bm{\phi}}}$ are constants. 

\subsection{Maximal isotropic foliations and null congruences}
We will be concerned with integrable distributions $\D\subset T_\C M$
that are maximal and isotropic, i.e., in $2m$ dimensions, $\D$ will be
$m$-dimensional and the metric vanishes on restriction to $\D$, i.e.,
for $\bm{V},\bm{W}\in\D$, $\bm{g}(\bm{V},\bm{W})=0$ and the integrability being given by the
Frobenius integrability condition $[\D,\D]\subset \D$.  It is always
possible to find a frame of 1-forms $\{\bm{\theta}^a \} = \{ \bm{\theta}^\mu,\bm{\theta}_\mu \}$,
($a=1,\ldots,2m$; $\mu=1,\ldots m$) such that $\D$ is the annihilator of the $\bm{\theta}_\mu$
and $\bm{g}= \sum_\mu \bm{\theta}^\mu\odot\bm{\theta}_\mu$.  

In Euclidean signature, $\D\cap\bar\D=\{0\}$ because there are no real
non-zero null vectors, and so such distributions correspond to complex
structures with respect to which the metric is Hermitian, i.e., we can
choose $\bm{\theta}^\mu=\overline{\bm{\theta}_\mu}$.

In Lorentz signature, $\D\cap\bar\D$ is $1$-dimensional because, on the
one hand, there are no linear subspaces of a Lorentzian lightcone of
dimension greater than one, and on the other, if $\D\cap
\bar\D=\{0\}$, $\D$ would be a complex structure for which the metric
is Hermitian, but such metrics must have an even number of positive
and negative eigenvalues over the reals, whereas in Lorentz signature
there is just one positive eigenvalue.  

We have the following lemma
\begin{lem}\label{total-geodesic}
  Suppose that the maximal isotropic distribution $\D$ is integrable,
  $[\D,\D]\subset \D$, then, supposing the space-time to be analytic,
  the integral surfaces of $\D$ in the complexification are totally
  geodesic.
\end{lem}

\begin{proof}
This can be seen as follows. Introduce a basis $\{\bm{V}_a\} =\{\bm{V}_\mu,\bm{V}^\mu\}$
of vector field dual to $\{\bm{\theta}^a\}=\{\bm{\theta}^\mu, \bm{\theta}_\mu\}$ where $\bm{V}_\mu$ spans
$\D$.  The Ricci rotation coefficients $\omega \ind{_{\mu \nu \lambda}}$,
$\omega \ind{_{\mu\nu}^\lambda}$, etc..., will
be given by $[\bm{V}_\mu,\bm{V}_\nu]= \sum_\lambda ( \omega \ind{_{\mu \nu \lambda}} \bm{V}^\lambda +
\omega \ind{_{\mu \nu}^\lambda} \bm{V}_\lambda)$. The integrability condition
$[\bm{V}_\mu,\bm{V}_\nu]\in\D$ implies that the Ricci rotation coefficients
$\omega_{\mu\nu\lambda}=0$ and so also the corresponding connection
coefficients $\Gamma_{\mu \nu \lambda}=0$ where $\nabla_{\bm{V}_\mu}
\bm{V}_\nu=\sum_\lambda (\Gamma \ind{_{\mu\nu} ^\lambda} \bm{V}_\lambda + \Gamma_{\mu\nu\lambda} \bm{V}^\lambda)$.
This allows one to deduce that the form 
$\bm{\theta}_1 \wedge \ldots \wedge \bm{\theta}_m$ ,
which is orthogonal to all the
$\bm{V}_\mu$, is parallel up scale along the $\bm{V}_\mu$, i.e.
\begin{align*}
\nabla_{\bm{V}_\mu} \bm{ \theta}_1 \wedge \ldots \wedge \bm{\theta} _m =
\sum_\nu \Gamma \ind{_{\mu\nu}^\nu} \bm{\theta}_1 \wedge \ldots \wedge \bm{\theta}_m.
\end{align*}  
Thus the integral surfaces of $\D$ are totally geodesic.
\end{proof}

We have the straightforward corollary
\begin{cor}
  In Lorentzian signature, the null congruence defined by
  $\D\cap\bar\D$ is geodesic.
\end{cor}

\subsection{The normal form of a generic 2-form}
Throughout this paper we will restrict attention to the case where the
Killing-Yano tensor $\bm{\phi}$ is generic in the sense that all its eigenvalues
will be assumed to be distinct (i.e., the eigenvalues of the
endomorphism obtained by raising one index with the metric).  The
following is a standard result and we only sketch its proof briefly to
set notation.

\begin{lem}\label{normform}
There exists a basis of 1-forms
$\{ \bm{\theta}^a\}=\{ \bm{\theta}^\mu, \bm{\theta}_\mu, \epsilon \bm{e}_{2m+1} \}$ 
that are a null basis for the metric, i.e.,
$\bm{g} = \sum_\mu
\bm{\theta}^\mu \odot \bm{\theta}_\mu + \epsilon \bm{e}_{2m+1}^2$ (so that each of the 1--forms
$\{\bm{\theta}^\mu, \bm{\theta}_\mu \}$ is null) 
and such that 
\begin{equation}\label{normal-form}
\bm{\phi}=\sum_\mu \lambda_\mu
\bm{\theta}^\mu\wedge \bm{\theta}_\mu \, .
\end{equation}
This normal form is unique up to separate rescalings of the
$(\bm{\theta}^\mu,\bm{\theta}_\mu) \rightarrow
(a_\mu \bm{\theta}^\mu, a_\mu^{-1}\bm{\theta}_\mu)$ with $a_\mu\neq 0$, and up to
permutations of the $\mu$ and $(\bm{\theta}^\mu, \bm{\theta}_\mu) \rightarrow
(\bm{\theta}_\mu, \bm{\theta}^\mu)$ for one value of $\mu$, with the other forms
left invariant.

[Here as before, $\epsilon=1$ in the odd dimensional case and zero otherwise.]
\end{lem}

\begin{proof}
[Sketch]
The genericity assumption allows us to use a basis of
eigen-(co)vectors for $\bm{\phi}$.  It is a standard fact that the
eigenvectors 
with non-zero eigenvalue of a 2-form with respect to a metric are
isotropic.  This follows from the identity (using the summation convention
now until the end of this section)
$$
\phi_{ab}l^b=\lambda g_{ab}l^b \qquad \Rightarrow \qquad  \lambda
l_al^a=\phi_{ab}l^al^b=0\, .
$$  
A pair of eigenvectors $l^a$, $n^a$ with eigenvalues $\lambda$, $\nu$
respectively satisfy $\phi_{ab}l^an^b=\nu l^an_a=-\lambda l^an_a$ and
so will be orthogonal to each-other unless their eigenvalues differ by
a sign.  Thus, since the eigenvectors span the space, they must come
in pairs with eigenvalues of opposite sign with possibly one zero
eigenvalue in odd dimensions, and can be normalised such
that the claims of the lemma are satisfied.
\end{proof}

We note that for a real $\bm{\phi}$ in Euclidean signature we must have
$\bm{\theta}^\mu=\overline {\bm{\theta}_\mu}$ and the $\lambda_\mu$ will all be
imaginary.  For real $\bm{\phi}$ in Lorentzian signature, we must have that
one eigenvalue, say $\lambda_1$ is real, as will therefore be
$\bm{\theta}^1$ and $\bm{\theta}_1$, with the other eigenvalues imaginary (with
$\bm{\theta}^\mu=\overline{\bm{\theta}_\mu} $).  This follows from the
requirement that the metric have just one timelike direction as before
in the discussion of maximal isotropic foliations.  

\section{Main results}
\subsection{Maximal isotropic distributions associated to $\bm\phi$
  and their integrability} 
Given the representation of $\bm{\phi}$ in Lemma \ref{normform}, we
can write down $2^m$ maximal isotropic distributions, one for each
choice of integers $\mu_1 <\mu_2 < \ldots <\mu_r \subset
\{1,\ldots,m\}$.  These are defined to be the distribution annihilated
by the forms $\{ \bm{\theta}^{\mu_1}, \ldots , \bm{\theta}^{\mu_r},
\bm{\theta}_{\nu_1},\ldots, \bm{\theta}_{\nu_{m-r}},\epsilon
\bm{e}_{2m+1} \}$ where the $\nu_1,\ldots,\nu_{m-r}$ are the distinct
integers in the complement of the $\mu_1,\ldots,\mu_r$ in
$\{1,\ldots,m\}$.  The purpose of this section is to prove the
following theorem.

\begin{thm} \label{eventhm}
  Let $M$ be a $2m$-dimensional Riemannian manifold equipped with
  a non-degenerate diagonalisable conformal Killing-Yano tensor $\bm{\phi}$
  with distinct eigenvalues. Let the $3$-form $\bm{\tau} = \dd
  \bm{\phi}$ satisfy 
\begin{align} \label{tauintcond}
	\bm{\theta}^\mu \wedge \bm{\tau} (\bm{V}^\mu, \cdot, \cdot) & = 0, &
	\bm{\theta}_\mu \wedge \bm{\tau} (\bm{V}_\mu, \cdot, \cdot) & = 0,
\end{align}
for each $\mu$ (i.e., with no summation).  Then the $2^m$ maximal
isotropic distributions associated to $\bm{\phi}$ are integrable. In
Euclidean signature they define $2^m$ distinct complex structures,
whereas in Lorentzian signature they define just two geodesic
congruences, each associated with $2^{m-1}$ integrable maximal
isotropic distributions.
\end{thm}

\begin{rem} Condition \eqref{tauintcond} is automatically satisfied 
	\begin{enumerate}
		\item when the conformal Killing-Yano tensor is closed, and
		\item in four dimensions.
	\end{enumerate}
\end{rem}

\begin{proof}
Let $M$ be a (real) $2m$-dimensional
(pseudo-) Riemannian manifold 
and $\bm{\phi}$ be a non-degenerate 
Killing-Yano tensor on $M$. Suppose that $\bm{\phi}$ has the form
\eqref{normal-form} in a null basis of one-forms
$\left\{\bm{\theta}^\mu,\bm{\theta}_\mu\right\}$, dual basis
$\{\bm V_\mu,\bm V^\mu\}$ and distinct 
eigenvalues $\left\{ \lambda_\mu, - \lambda_\mu \right\}$.

In terms of the vector basis $\left\{ \bm{V}_\mu, \bm{V}^\mu
\right\}$, the integrability of all these distributions will be implied 
by the conditions
\begin{align} \label{Riccirot}
	[ \bm{V}_\mu, \bm{V}_\nu ] & = \omega \ind{_{\mu} _{\nu} ^{\mu}} \bm{V}_\mu + \omega \ind{_{\mu} _{\nu} ^{\nu}} \bm{V}_\nu, \nonumber \\
	[ \bm{V}^\mu, \bm{V}^\nu ] & = \omega \ind{^{\mu} ^{\nu} _{\mu}} \bm{V}^\mu + \omega \ind{^{\mu} ^{\nu} _{\nu}} \bm{V}^\nu, \\
	[ \bm{V}_\mu, \bm{V}^\nu ] & = \omega \ind{_{\mu} ^{\nu} ^{\mu}} \bm{V}_\mu + \omega \ind{_{\mu} ^{\nu} _{\nu}} \bm{V}^\nu, \nonumber
\end{align}
satisfied for all distinct $\mu$, $\nu$, and no summation.   These are 
constraints on the Ricci rotation coefficients and hence on the connection.
In terms of the connection coefficients
 we must show
\begin{align} \label{int2m}
  \Gamma \ind{_{\kappa \mu \nu}} & = 0, &
  \Gamma \ind{^{\kappa \mu \nu}} & = 0, \qquad (\mbox{for all $\kappa, \mu,\nu$}), \nonumber \\
  \Gamma \ind{_{\kappa \mu} ^\nu} & = 0, &
  \Gamma \ind{^\kappa _\nu ^\mu} & = 0, \qquad (\mbox{for all $\nu \neq \mu, \kappa$})
  , \\
  \Gamma \ind{_\kappa ^{\mu \nu}} & = 0, &
  \Gamma \ind{^\kappa _{\mu \nu}} & = 0, \qquad (\mbox{for all $\kappa \neq \mu,\nu$}). \nonumber
\end{align}
which imply equations \eqref{Riccirot}.

In terms of basis components, the Killing-Yano equation \eqref{CKYeq} yields (no summation)
\begin{subequations}
\begin{align}
  \partial \ind{_{\kappa}} \phi \ind{_\mu ^\nu} + \left( \lambda_\mu -
  \lambda_\nu
  \right) \Gamma \ind{_{\kappa \mu} ^\nu} & = \tau \ind{_{\kappa \mu} ^\nu} - \frac{1}{n-1} \delta
  \ind*{_\kappa
  ^\nu} K \ind{_\mu},  \label{eq_*K1}
%&  \partial \ind{^\kappa} \phi \ind{_\mu ^\nu} + \left( \lambda_\mu -
%  \lambda_\nu
%  \right) \Gamma \ind{^\kappa _\mu ^\nu} & = \tau \ind{^{\kappa} _\mu ^\nu} + \frac{1}{n-1} \delta
%  \ind*{^\kappa
%  _\mu} K \ind{^\nu},
\\
  \left( \lambda_\mu + \lambda_\nu \right) \Gamma \ind{^\kappa _{\mu \nu}} & = \tau \ind{^\kappa _\mu _\nu} +\frac{2}{n-1} \delta
  \ind*{^\kappa _{\lb{\mu}}} K \ind{_{\rb{\nu}}}\label{eq_*K2},
%& - \left( \lambda_\mu + \lambda_\nu \right) \Gamma \ind{_\kappa ^{\mu \nu}} & = \tau \ind{_\kappa ^\mu ^\nu} + \frac{2}{n-1} \delta
%  \ind*{_\kappa ^{\lb{\mu}}} K \ind{^{\rb{\nu}}},
\\
  \left( \lambda_\mu + \lambda_\nu \right) \Gamma \ind{_{\kappa \mu \nu}} & =
  \tau \ind{_{\kappa \mu \nu}}.  \label{eq_*K3}
%&
%  \left( \lambda_\mu + \lambda_\nu \right) \Gamma \ind{^{\kappa \mu \nu}} & =
%  \tau \ind{^{\kappa \mu \nu}}.
\end{align}
\end{subequations}
 From equations \eqref{eq_*K1}, we then obtain
\begin{subequations}
\begin{align}
  K _\mu & = - (n-1) \partial _\mu \lambda _\mu,
%&  K ^\mu & = (n-1) \partial ^\mu \lambda
%    _\mu,
  \label{compKV1}\\
  \tau \ind{_{\nu \mu} ^\mu} & = \partial \ind{_\nu} \lambda \ind{_\mu},
%& \tau \ind{^{\nu \mu} _\mu} & =\partial \ind{^\nu} \lambda
%  \ind{_\mu},
\label{consteval}
& \mbox{for all $\nu \neq \mu$},\\
   \tau \ind{_{\nu \mu} ^\nu} - \frac{1}{n-1}  K \ind{_\mu} & = (\lambda \ind{_\mu} - \lambda \ind{_\nu})
  \Gamma \ind{_{\nu \mu} ^\nu}, 
%&
%  \tau \ind{_^\nu _\nu ^\mu} - \frac{1}{n-1} K \ind{^\mu} & = (\lambda \ind{_\mu} - \lambda \ind{_\nu})
%  \Gamma \ind{^\nu _\nu ^\mu},
\label{compKV2} & \mbox{for all $\nu \neq \mu$},\\
 \delta \ind*{_\kappa ^\mu} \tau \ind{_{\kappa \mu} ^\nu} & = \left( \lambda_\mu -
  \lambda_\nu
  \right) \Gamma \ind{_{\kappa \mu} ^\nu},
%&  \delta \ind*{_\kappa ^\mu} \tau \ind{^{\kappa} _\nu ^\mu} & = \left( \lambda_\mu -
%  \lambda_\nu
%  \right) \Gamma \ind{^\kappa _\nu ^\mu}, 
\label{PIntCond2} & \mbox{for all $\nu \neq \kappa, \mu$}.
\end{align}
\end{subequations}
On the other hand, equations
\eqref{eq_*K2} give
\begin{subequations}
\begin{align}
  \tau \ind{_\kappa ^\mu ^\nu} & = (\lambda_\mu + \lambda_\nu) \Gamma \ind{_\kappa ^\mu ^\nu},
%& \tau \ind{^\kappa _\mu _\nu} & = (\lambda_\mu + \lambda_\nu) \Gamma \ind{^\kappa _\mu _\nu},
& \mbox{for all $\kappa \neq
\mu,\nu$}, \label{PIntCond3} \\
  \tau \ind{^\nu _\nu _\mu} + \frac{1}{n-1} K \ind{_\mu} & = (\lambda \ind{_\mu} + \lambda \ind{_\nu})
  \Gamma \ind{^\nu _\nu _\mu},
%&  \tau \ind{_\nu ^\nu ^\mu} - \frac{1}{n-1} K \ind{^\mu} & = (\lambda \ind{_\mu} + \lambda \ind{_\nu})
%  \Gamma \ind{_\nu ^\nu ^\mu},
& \mbox{for all $\nu \neq \mu$}. \label{compKV3}
\end{align}
\end{subequations}
By symmetry, we have similar relations involving $\Gamma \ind{^\kappa _\mu ^\nu}$, $\Gamma \ind{_\kappa ^\mu ^\nu}$, and $\Gamma \ind{^\kappa ^\mu ^\nu}$. By assumption all the eigenvalues $\left\{ \lambda_\mu,
-\lambda_\mu \right\}$ are distinct, so that equations \eqref{eq_*K3},
\eqref{PIntCond2} and \eqref{PIntCond3} imply the integrability conditions \eqref{int2m}
if and only if
\begin{align*}
  \tau \ind{_{\kappa \mu \nu}} & = 0, &
  \tau \ind{^{\kappa \mu \nu}} & = 0, \qquad (\mbox{for all $\kappa, \mu,\nu$}), \\
  \tau \ind{_{\kappa \mu} ^\nu} & = 0, &
  \tau \ind{^\kappa _\nu ^\mu} & = 0, \qquad (\mbox{for all $\nu \neq \mu, \kappa$})
  , \\
  \tau \ind{_\kappa ^{\mu \nu}} & = 0, &
  \tau \ind{^\kappa _{\mu \nu}} & = 0, \qquad (\mbox{for all $\kappa \neq \mu,\nu$}),
\end{align*}
which is equivalent to equation \eqref{tauintcond}.

At this point, we now have enough information about the connection to
obtain the integrability of the maximal isotropic distributions.  In
particular, we have the condition $\Gamma_{\mu\nu\lambda}=0$ which
implies as in the proof of Lemma \ref{total-geodesic},
$$
\nabla_{\bm{V}_\mu} \bm{\theta}_1 \wedge \ldots \wedge \bm{\theta}_m=
\sum_\nu \Gamma \ind{_{\mu\nu}^\nu} \bm{\theta}_1 \wedge \ldots \wedge \bm{\theta}_m .
$$  
This, in particular, implies for
$\bm{\Omega}= \bm{\theta}_1 \wedge \ldots \wedge \bm{\theta}_m$ 
that $\dd \bm{\Omega}= \bm{\alpha} \wedge \bm{\Omega}$ for some $1$-form $\bm{\alpha}$, so that the distribution
$\D=\langle \bm{V}_1, \ldots, \bm{V}_m \rangle$
orthogonal to $\bm{\Omega}$ is integrable.  However, all the maximal
isotropic distributions determined by $\bm{\phi}$ are on an equal footing
with $\D$; all such distributions are equivalent to $\D$ by
interchanging
$( \bm{\theta}^\mu, \bm{\theta}_\mu)\rightarrow(\bm{\theta}_\mu,\bm{\theta}^\mu)$ for
different values of $\mu$. 
Thus, all such distributions are integrable.  This can, of course, be
checked explicitly by calculating $\nabla_{\bm{W}} \bm{\theta}_{\sigma
  (1)} \wedge \ldots \wedge \bm{\theta}_{\sigma(p)}\wedge \bm{\theta}^{\sigma(p+1)}\wedge\ldots
\bm{\theta}^{\sigma(m)}$ where $\sigma$ is an arbitrary permutation of
$1,\ldots,m$ and $\bm{W}$ is in the kernel of $\bm{\theta}_{\sigma
  (1)}\wedge\ldots\wedge \bm{\theta}_{\sigma(p)} \wedge \bm{\theta}^{\sigma(p+1)}\wedge\ldots
\bm{\theta}^{\sigma(m)}$.
\end{proof}

For future use we record the expressions for some of the remaining connection components. Combining equations \eqref{compKV1},
\eqref{compKV2} and \eqref{compKV3} gives
\begin{subequations}
\begin{align}
  \Gamma \ind{_\nu _\mu ^\nu}  & = \partial _\mu  \ln | \lambda_\mu - \lambda_\nu |,
&  \Gamma \ind{^\nu ^\mu _\nu}  & =  \partial ^\mu  \ln | \lambda_\mu - \lambda_\nu |,\label{LC-}  \\
  \Gamma \ind{^\nu _\mu _\nu}  & =  \partial _\mu \ln  | \lambda_\mu + \lambda_\nu |,
&  \Gamma \ind{_\nu ^\mu ^\nu}  & =  \partial ^\mu \ln  | \lambda_\mu + \lambda_\nu |. \label{LC+}
\end{align}
\end{subequations}

\begin{rem}
  We emphasise that the above result is essentially local. We have
  made the assumption, for example, that the Killing-Yano tensor has
  distinct and non-constant eigenvalues, and this assumption will
  generically break down on some nontrivial subset of codimension at
  least one. In general, then, the complex structures will not extend
  globally over such subsets. The Kerr-NUT-(A)dS metric provides such
  an example. Setting the mass and the NUT parameters to zero, the
  metric reduces to a Ricci-flat conformally flat metric. But it is a
  standard result that the round four-sphere does not admit a global
  hermitian complex structure (there are complex structures on
  the complement of a point in $S^4$ that naturally extend to $\C{\mathbb
    P}^2$ or a quadric). 
\end{rem}

\subsection{Odd-dimensional manifolds and integrable CR
  structures} 
\label{Oddimcase}

The above results extend naturally to odd-dimensional manifolds. When
$M$ is a $(2m+1)$-dimensional real manifold, the natural analogue of a
complex structure is a CR structure.  An almost Cauchy-Riemann (CR)
structure is an $m$-dimensional subbundle $\D$ of the complexified
tangent bundle $T_\mathbb{C} M$, so that $\D\cap\bar \D=0$.  It is a
CR structure if it is integrable, $[\D,\D]\subset \D$.  This is the
structure that a hypersurface in $\C^{m+1}$ inherits from the ambient
complex structure; $\D$ are those vectors in the holomorphic tangent bundle $T^{(1,0)}$ that are
tangent to the hypersurface.  On our (pseudo-) Riemannian manifold, we
will also require $\D$ to be isotropic so that we will have
$T_\mathbb{C} M = \D \oplus \bar\D \oplus K$ where $K$ is the
orthogonal complement of $\D\oplus\bar\D$.

Given a non-degenerate Killing-Yano tensor $\bm{\phi}$ on $M$, $K$
will be the kernel of $\bm{\phi}$. By
Lemma \ref{normform}, assuming that $\bm{\phi}$ has distinct non-zero
eigenvalues, we can find a basis
of $1$-forms $\left\{ \bm{e}^{0},\bm{\theta}^\mu, \bm{\theta}_\mu
\right\}$ all null except $\bm e^0$, in which $\bm{\phi}=\sum_\mu \lambda_\mu
\theta^\mu\wedge\theta_\mu$  is diagonal, degenerate on $K$, but
non-degenerate on $K^\perp$.
 We write $\left\{  \bm{V}_{0} 
, \bm{V}_\mu, \bm{V}^\mu
\right\}$, ($\mu=1,\ldots,m$), for the corresponding
dual vector basis so that $\bm{V}_0$ spans $K$, and the $2^m$ distributions
are found by choosing one vector from each of the pairs $(\bm{V}_{\mu}
,\bm{V}^{\mu} )$ for each $\mu$.
  
The question of whether the $2^m$ almost Cauchy-Riemann structures
$\bm{\phi}$ are integrable reduces to the integrability of the $2^m$
maximal isotropic distributions on $K^\perp$, hence, to whether
relations \eqref{Riccirot} are satisfied. So, if $\D$ is one of our
$2^m$ maximal
isotropic distribution in $K^\perp$, then $[ \D, \D ] \subset \D$.
More precisely, relations
\eqref{Riccirot} tell us that $\omega \ind{_{\mu} _{\nu}_0}$, $\omega
\ind{^{\mu} ^{\nu} _0}$, and $\omega \ind{_{\mu} ^{\nu} _0}$ vanish
for all $\mu \neq \nu$. 
In odd dimensions, the conformal Killing-Yano
equations \eqref{CKYeq} gives the extra conditions
\begin{subequations}
\begin{align}
	\partial_0 \phi \ind{_\mu ^\nu} + (\lambda_\mu - \lambda_\nu) \Gamma \ind{_0 _\mu ^\nu} & = 0 \label{CKYodd1a}\\
	(\lambda_\mu + \lambda_\nu) \Gamma \ind{_0 _\mu _\nu} & = \tau \ind{_0 _\mu _\nu} \label{CKYodd1b}\\
	 \lambda_\mu \Gamma \ind{_0 _0 _\mu}  & = \frac{1}{n-1} K \ind{_\mu} \label{CKYodd1c}\\
	-\lambda_\mu \Gamma \ind{_{\nu \mu 0}}  & = \tau \ind{_{\nu \mu 0}} \label{CKYodd1d} \\
	\lambda_\mu \Gamma \ind{^\nu _{\mu 0}} & = \tau \ind{^\nu _{\mu 0}} + \frac{1}{n-1} \delta \ind*{_\mu ^\nu} K_0. \label{CKYodd1e}
\end{align}
\end{subequations}
Equations \eqref{CKYodd1a} lead to
\begin{subequations}
\begin{align}
	\partial_0 \lambda_\mu & =0 \label{CKYodd2a} \\
	(\lambda_\mu - \lambda_\nu) \Gamma \ind{_0 _\mu ^\nu} & = 0 & \mbox{for all $\mu \neq \nu$}, \label{CKYodd2b}
\end{align}
\end{subequations}
and
equations \eqref{CKYodd1e} to
\begin{subequations}
\begin{align}
	\lambda_\mu \Gamma \ind{^\nu _{\mu 0}} & = \tau \ind{^\nu _{\mu 0}} & \mbox{for all $\mu \neq \nu$}, \label{CKYodd3a} \\
	\lambda_\mu \Gamma \ind{^\mu _{\mu 0}} & = \tau \ind{^\mu _{\mu 0}} + \frac{1}{n-1} K_0. \label{CKYodd3b}
\end{align}
\end{subequations}
By the assumptions on the eigenvalues, equations \eqref{CKYodd1d} and \eqref{CKYodd3a} show that
\begin{align*}
	\Gamma \ind{_{\nu \mu 0}} & = 0 = \Gamma \ind{^{\nu \mu} _0} = \Gamma \ind{^\nu _{\mu 0}} & \Longleftrightarrow  & & \tau \ind{_{\nu \mu 0}} & = 0 =  \tau \ind{^{\nu \mu} _0} = \tau \ind{^\nu _{\mu 0}}, & \mbox{for all $\mu \neq \nu$},
\end{align*}
which is subsumed into condition \eqref{tauintcond}.
Thus, Theorem \ref{eventhm} extends to the odd-dimensional case.

\begin{rem}
  Given such an integrable distribution $\D$, we can adjoin $\bm V_0$
  to form the distribution $\tilde \D= \{\bm V_0, \D\}$.  The
  integrability of this distribution requires in addition to the
  above, that $[\bm V_\mu,\bm V_0]=\omega \ind{_{\mu 0}^\mu} \bm V_\mu +
  \omega \ind{_{\mu 0}^0} \bm V_0$ and $[\bm V^\mu,\bm V_0]=\omega \ind{^{\mu} _{0 \mu}} \bm V^\mu +
  \omega \ind{^\mu _0^0} \bm V_0$ for all $\mu$, but this follows from the above
  conditions on the connection.  Thus, these distributions will also
  be integrable.
\end{rem}

We also note that combining equations \eqref{compKV1} and  \eqref{CKYodd1c} yields
\begin{align}\label{OddCond}
	\Gamma \ind{_{0 0} _\mu} = - \partial_\mu \ln |\lambda_\mu|.
\end{align}

\subsection{Examples}

\subsubsection{The $*$-Killing (or closed conformal Killing-Yano) case}
For the Kerr-NUT-(A)dS, we can see these distributions explicitly.
The integrability is most evident in terms of the inverse metric which
in the even dimensional case is given by \cites{Chen2006}
$$
\bm{g}^{-1} =\sum_{\mu=1}^m 2 \bm{V}^\mu \odot \bm{V}_\mu
$$
where
\begin{align*}
\bm{V}_\mu & = \frac{2^{-1/2}}{\sqrt{U_\mu}}\left( \sqrt{X_\mu}\parderv{}{x_\mu} -
\ii
\frac{1}{\sqrt{X_\mu}}\left(
\sum_{k=0}^{m-1} (-1)^k x_\mu ^{2(m-1-k)}\parderv{}{\psi_k} \right) \right)\, , &
\bm{V}^\mu & =\overline{\bm{V}_\mu}.
\end{align*}
The key feature here is that, apart from the scale factor $\sqrt{U_\mu}$, in
these coordinates, the coefficients of the vectors $\bm{V}^\mu$ and $\bm{V}_\mu$
depend only on the coordinate $x_\mu$.  Thus a distributions made up
of any set of the basis vectors with distinct values of $\mu$ will
commute amongst themselves so that the distribution is integrable.
(The only pairs of these basis vectors that do no commute amongst
themselves in this way are $\{\bm{V}^\mu, \bm{V}_\mu\}$.) In
particular, all the maximal istropic distributions spanned by $m$ such
basis vectors with distinct values of $\mu$ will form an integrable
distribution.  
Explicitly, we have
\begin{align*}
	\left[ \bm{V}_\mu , \bm{V}_\nu \right] & = 2^{-1/2} \frac{x_\nu \sqrt{Q_\nu}}{x_\nu ^2 - x_\mu^2} \bm{V} _{\mu} - 2^{-1/2} \frac{x_\mu \sqrt{Q_\mu}}{x_\mu ^2 - x_\nu^2} \bm{V} _{\nu} \\
	\left[ \bm{V}^\mu , \bm{V}^\nu \right] & =  2^{-1/2} \frac{x_\nu \sqrt{Q_\nu}}{x_\nu ^2 - x_\mu^2}  \bm{V}^\mu - 2^{-1/2} \frac{x_\mu \sqrt{Q_\mu}}{x_\mu ^2 - x_\nu^2} \bm{V}^\nu  \\
	\left[ \bm{V}_\mu , \bm{V}^\nu \right] & = 2^{-1/2} \frac{x_\nu \sqrt{Q_\nu}}{x_\nu ^2 - x_\mu^2} \bm{V} _\mu - 2^{-1/2} \frac{x_\mu \sqrt{Q_\mu}}{x_\mu ^2 - x_\nu^2} \bm{V}^\nu  \\
	\left[ \bm{V}_\mu , \bm{V}^\mu \right] & = 2^{-1/2}  \parderv{\sqrt{Q_\mu}}{x_\mu} \left( \bm{V} _\mu  - \bm{V}^\mu \right) + 2 \cdot 2^{-1/2} \sum_{\nu \neq \mu} \frac{x_\mu \sqrt{Q_\nu}}{x_\nu ^2 - x_\mu^2}\left( \bm{V} _\nu  - \bm{V}^\nu \right) ,
\end{align*}
where $Q_\mu = X_\mu / U_\mu$.

\subsubsection{A five-dimensional black hole solution}
We consider the recently discovered metric of L\"{u}, Mei and Pope, \cites{Lu2009}.
The metric is $\bm{g}_5=\sum_{i=0}^4 (\bm e^i)^2$ with
\begin{align*}
	\bm{e}^0 & = \left( \frac{a_0}{xy} \right)^{1/2} (\dd \phi +
        (x+y) \dd \psi + xy \dd t). 
\\
	\bm{e}^1 & = \frac{1}{2(1-x y)}\left(\frac{x-y}{X}\right)^{1/2} \dd x &
	\bm{e}^3 & = \frac{1}{1-xy}\left(\frac{X}{x(x-y)}\right)^{1/2} (\dd \phi +y \dd \psi) \\
 	\bm{e}^2 & = \frac{1}{2(1-x y)}\left(\frac{y-x}{Y}\right)^{1/2} \dd y &
	\bm{e}^4 & = \frac{1}{1-xy}\left(\frac{Y}{y(y-x)}\right)^{1/2} (\dd \phi +x \dd \psi) 
\end{align*}
where $X$ and $Y$ are quartic polynomials in $x$ and $y$
respectively.  Then, the dual vector basis is given by
\begin{align*}
	\bm{e}_0 & = \left( \frac{1}{a_0 xy} \right)
        ^{1/2} \partial_t\\
	\bm{e}_1 & = 2 (1-xy)\left( \frac{X}{x-y} \right) ^{1/2} \partial_{x} &
	\bm{e}_3 & = (1-xy) \left( \frac{1}{x(x-y)X} \right)^{1/2} (x^2 \partial_{\phi} - x \partial_\psi + \partial_t) \\
	\bm{e}_2 & = 2 (1-xy)\left( \frac{Y}{y-x} \right) ^{1/2} \partial_y &
	\bm{e}_4 & = (1-xy) \left( \frac{1}{y(y-x)Y} \right)^{1/2} (y^2 \partial_\phi - y \partial_\psi + \partial_t)
\end{align*}
Defining $\bm{V}_1 = 2^{-1/2}(\bm{e}_1 - \ii \bm{e}_3)$ and
$\bm{V}_2 = 2^{-1/2}(\bm{e}_2 - \ii \bm{e}_4)$, we obtain the
Ricci rotation coefficients from the commutators
\begin{align*}
	\left[ \bm{V}_1, \bm{V}_2 \right] & = -2^{-1/2} \left( \frac{Y}{y-x} \right) ^{1/2} \frac{ (-2x^2 + xy +1)}{x-y} \bm{V}_1
+ 2^{-1/2}\left( \frac{X}{x-y} \right) ^{1/2} \frac{(-2y^2 + xy +1) }{y-x} \bm{V}_2 \\
	\left[ \bm{V}_1, \bm{V}^2 \right] & = - 2^{-1/2}\left( \frac{Y}{y-x} \right) ^{1/2} \frac{ (-2x^2 + xy +1)}{x-y} \bm{V}_1
+ 2^{-1/2}\left( \frac{X}{x-y} \right) ^{1/2} \frac{(-2y^2 + xy +1) }{y-x} \bm{V}^2 \\
	\left[ \bm{V}_1, \bm{V}^1 \right] & = 2^{-1/2} \left( 2 \partial_x \left( (1-xy)\left(\frac{X}{x-y}\right)^{1/2}\right) +\left(\frac{X}{x-y}\right)^{1/2}  \left( \frac{1-5xy}{x} \right) \right) (\bm{V}_1 - \bm{V}^1) \\
& \qquad -  2^{-1/2}\left( \frac{Y}{y-x} \right) ^{1/2} \left( \frac{x}{y} \right) ^{1/2} \frac{2(1-xy) }{(x-y)} (\bm{V}_2 - \bm{V}^2) 
+ 2 \ii  \left(\frac{a_0}{xy}\right)^{1/2} (1-xy)^2 x^{-1/2} \bm{e}_0 \\
	\left[ \bm{V}_1, \bm{e}_0 \right] & = - 2^{-1/2}
        \frac{1-xy}{x} \left( \frac{X}{x-y} \right)^{1/2} \bm{e}_0 
\end{align*}
with the remaining commutators given by complex conjugation and the
symmetry $ 1\leftrightarrow 2$ accompanied by $x\leftrightarrow y$ and
$X\leftrightarrow Y$. 

It is also straighforward to 
check that its associated rank-two and rank-three complex distributions 
are all integrable .

We now turn our attention to the existence of a conformal Killing-Yano tensor $\bm{\phi}$ in normal form in this basis. We first note that the four-dimensional metric $\bm{g}_4 = \sum_{i=1}^4 (\bm e^i)^2$ has a conformal Killing-Yano tensor given by
\begin{align*}
	\bm{\phi} & = \ii \frac{x^{1/2}}{1-xy} \bm{\theta}^1 \wedge \bm{\theta} _1 +
        \ii \frac{y^{1/2}}{1-xy} \bm{\theta}^2 \wedge \bm{\theta} _2 \\
 & =  \frac{1}{2 (1-xy)^3} ( \dd x \wedge (\dd \phi + y \dd \psi) + \dd y \wedge (\dd \phi + x \dd \psi) )\,
\end{align*}
where $\bm{\theta}^\mu = 2^{-1/2}(\bm{e}^\mu + \ii \bm{e}^{2+\mu})$ and $\bm{\theta}_\mu = \overline{\bm{\theta}^\mu}$. This choice can be justified by the fact that the metric $\bm{g}_4$ is simply a conformal rescaling of a (euclideanised) Kerr metric in Pleba\'{n}ski-Demia\'{n}ski form, and that a conformal Killing-Yano tensor has conformal weight $3$. One can also check that, with this choice, the eigenvalues of $\bm{\phi}$ satisfy equations \eqref{LC-} and \eqref{LC+}. On the other hand, on considering the full metric metric $\bm{g}_5$ and by \eqref{OddCond}, the eigenvalues $\lambda_\mu$ of $\bm{\phi}$ must also satisfy
\begin{align*}
	2 \partial_x \ln |\lambda_1| & = \frac{1}{x} & 	2 \partial_y \ln |\lambda_2| & = \frac{1}{y},
\end{align*}
solutions of which can be taken to be $\lambda_1 = x^{1/2}$ and $\lambda_2 = y^{1/2}$.
Hence, in spite of the existence of integrable maximal isotropic distributions, there is no conformal Killing-Yano tensor in normal form in this basis, and the converse of the (odd-dimensional version of) Theorem \ref{eventhm} does not hold.

\subsection{Conditions on the Weyl conformal tensor}
Let us return to the general case of a conformal Killing-Yano tensor
$\bm{\phi}$ on a real $2m$-dimensional (pseudo-) Riemannian 
manifold $M$. As before, we consider the complexification of the
tangent bundle $T M$. We can extend $\bm{\phi}$ to an endomorphism
$\hat{\bm{\phi}}$ on the space of $2$-forms $\bigwedge ^2 T^*
M$. Similarly, we can view the Weyl tensor $\bm{C}$ as an endomorphism
$\hat{\bm{C}}$ on $\bigwedge ^2 T^* M$. It is a standard result
\cite{Semmelmann2001} that the commutator of 
$\hat{\bm{C}}$ and $\hat{\bm{\phi}}$ vanishes, i.e.
\begin{align} \label{WTws2}
        \left[ \hat{\bm{C}} , \hat{\bm{\phi}} \right] = 0.
\end{align}
If $\bm{\phi}$ is diagonal in the null basis $\left\{ \bm{\theta} ^a \right\}=\left\{ \bm{\theta} ^\mu , \bm{\theta} _\mu \right\}$, then $\hat{\bm{\phi}}$ is also diagonal in the canonical basis of $2$-forms $\left\{ \bm{\theta} \ind{^a}\wedge \bm{\theta} \ind{^b} \right\}$, and
\begin{align*}
        \hat{\bm{\phi}} \left(\bm{\theta} \ind{^\mu} \wedge \bm{\theta} \ind{^\nu} \right) & = \left(\lambda _{\mu} + \lambda _{\nu} \right) \bm{\theta} \ind{^\mu} \wedge \bm{\theta} \ind{^\nu}, &
\hat{\bm{\phi}} \left(\bm{\theta} \ind{_\mu} \wedge \bm{\theta} \ind{_\nu} \right) & = -\left( \lambda _{\mu} + \lambda _{\nu} \right) \bm{\theta} \ind{_\mu} \wedge \bm{\theta} \ind{_\nu}, \\
        \hat{\bm{\phi}} \left(\bm{\theta} \ind{^\mu} \wedge \bm{\theta} \ind{_\nu} \right) & = \left( \lambda _{\mu} - \lambda _{\nu} \right) \bm{\theta} \ind{^\mu} \wedge \bm{\theta} \ind{_\nu}.
\end{align*}
Assuming that the eigenvalues of $\bm{\phi}$ are all distinct, $\hat{\bm{\phi}}$ has $2m(m-1)$ non-zero eigenvalues and has an $m$-dimensional kernel spanned by $\left\{ \bm{\theta} \ind{^\mu} \wedge \bm{\theta} \ind{_\mu} : \mu =1, \ldots m \right\}$. By the commutation relation \eqref{WTws2}, it then follows that $\hat{\bm{C}}$ and $\hat{\bm{\phi}}$ have $m(2m-1)$ common eigen-$2$-forms,
and we can deduce that all components of the Weyl tensor with the possible exception of
\begin{align*}
	C \ind{_{\mu \nu} ^{\mu \nu}}, && C \ind{_\mu ^\nu _\nu ^\mu }, && C \ind{_\mu ^\mu _\nu ^\nu}, && C \ind{_\mu ^\mu _\mu ^\mu},
\end{align*}
for all distinct $\mu$, $\nu$, vanish in the canonical basis. Consequently,
\begin{thm} \label{TWeyl}
        Let $\bm{\phi}$ be a non-degenerate conformal Killing-Yano tensor with distinct eigenvalues, diagonal in the null basis $\left\{ \bm{\theta} ^a \right\}=\left\{ \bm{\theta} ^\mu , \bm{\theta} _\mu \right\}$. Then the Weyl tensor $\bm{C}$ satisfies (no summation)
\begin{align*}
        \bm{C} (\bm{V} _a, \bm{V} _b, \bm{V}_c, \bm{V} _d) & = 0, & 
\bm{C} (\bm{V} _\mu, \bm{V} ^\mu, \bm{V} _a, \bm{V} _b) & = 0,\\
	 \bm{C} (\bm{V} _\mu, \bm{V} _a, \bm{V} ^\mu, \bm{V}_b) & = 0, & \bm{C} (\bm{V} _\mu, \bm{V} _a, \bm{V} _\mu, \bm{V}_b) & = 0,
\end{align*}
for all distinct $\mu,a,b,c,d$, where $\left\{ \bm{V} _a \right\}=\left\{ \bm{V} _\mu , \bm{V} ^\mu \right\}$ is the dual basis.
%Alternatively,
%\begin{align*}
%        \bm{C} (\bm{V} _a, \bm{V} _a) \wedge \bm{V} _a & = 0
%\end{align*}
%for all $a$, where $\bm{C}$ is now viewed as an element of $\bigodot ^2 \mathfrak{so}(n)$.
\end{thm}

\begin{rem}
	Using the notation set in Section \ref{Oddimcase}, we note that 
%\begin{align*}
%\hat{\bm{\phi}} \left(\bm{\theta} \ind{^\mu} \wedge \bm{e} \ind{^0} \right) & = \lambda _{\mu} \bm{\theta} \ind{^\mu} \wedge \bm{\theta} \ind{^0} & \hat{\bm{\phi}} \left(\bm{\theta} \ind{_\mu} \wedge \bm{e} \ind{^0} \right) & = -\lambda _{\mu} \bm{\theta} \ind{_\mu} \wedge \bm{\theta} \ind{^0},
%\end{align*}
$\bm{\theta} \ind{^\mu} \wedge \bm{e} \ind{^0}$ and $\bm{\theta} \ind{_\mu} \wedge \bm{e} \ind{^0}$ are also eigenspinors of $\hat{\bm{\phi}}$, and thus, of $\hat{\bm{C}}$ too, which implies that $C \ind{_\mu _0 ^\mu ^0}$ may not vanish. Hence, the above theorem also holds in odd dimensions, where the lower-case Roman indices may now take the value $0$.
\end{rem}

In four dimensions, the Weyl tensor $\hat{\bm{C}}$
splits into a SD part $\hat{\bm{C}}^+$ and an ADS part
$\hat{\bm{C}}^-$. Each of $\hat{\bm{C}}^\pm$ has a pair of degenerate
eigenvalues, and their eigen-$2$-forms are precisely the SD and ASD
isotropic $2$-planes
\begin{align*}
\left\{ \bm{\theta} \ind{^1} \wedge \bm{\theta} \ind{^2},\, \bm{\theta}
  \ind{_1} \wedge \bm{\theta} \ind{_2},\, \bm{\theta} \ind{^1} \wedge \bm{\theta} \ind{ _1} + \bm{\theta} \ind{^2} \wedge \bm{\theta} \ind{_2}
\right\}
\quad \mbox{and} \quad \left\{ \bm{\theta} \ind{^1} \wedge \bm{\theta} \ind{_2},\, \bm{\theta} \ind{_1} \wedge \bm{\theta} \ind{^2},\, \bm{\theta} \ind{^1} \wedge \bm{\theta} \ind{_1} - \bm{\theta} \ind{^2} \wedge \bm{\theta} \ind{_2}
\right\}.
\end{align*}
In the language of general relativity, this is the defining
property for the manifold to be of Petrov\footnote{Although this
  classification applies mostly to 
Lorentzian manifolds, it can easily be extended to four-dimensional
proper Riemmanian 
manifolds \cite{Karlhede1986}.} type D. In fact, the existence of a conformal Killing $2$-form $\bm{\phi}$ on a four-dimensional (Lorentzian) manifold implies \cites{Dietz1981,Dietz1982,Glass1999} that the spacetime is of
Petrov type D or N according to whether $\bm{\phi}$ is of rank
$4$ or of rank $2$. A classification of the Weyl tensor in higher-dimensional Lorentzian spacetimes has been undertaken in 
\cites{Coley2004,Milson2005,Coley2006,Pravda2004,Pravda2007}, wherein the four-dimensional concept of \emph{(gravitational) principal null direction} (GPND) is generalised to that of \emph{Weyl aligned null directions} (WAND). It is shown in \cites{Hamamoto2007,Pravda2007} that the
Kerr-NUT-(A)dS metric is of Petrov type D in an
appropriate sense. More generally, the following statement, which
answers a conjecture put forward in \cite{Frolov2008}, is a direct
consequence of Theorem \ref{TWeyl}. 
\begin{cor}
        Let $\bm{\phi}$ be a non-degenerate conformal Killing-Yano tensor
        with distinct eigenvalues, diagonal in the null basis
        $\left\{ \bm{\theta} ^a \right\}$. Then each of the basis vectors $\{ \bm{V} _a \}$
        is a WAND of the Weyl tensor $\bm{C}$. In particular, $\bm{C}$
        is of type D. 
\end{cor}

\subsection{Relation to Hamiltonian $2$-forms}

Reference \cite{Apostolov2006} introduces the notion of a \emph{Hamiltonian
  $2$-form} on a K\"{a}hler manifold, a $(1,1)$-form $\bm{\psi}$ which satisfies 
\begin{align}\label{Hamil2}
        \nabla _{\bm{X}} \bm{\psi} & = \frac{1}{2} \left( \dd \sigma
          \wedge \bm{J} (\bm{X}^*) - \bm{J} (\dd \sigma) \wedge
          \bm{X}^* \right), & \Biggl( \nabla \ind{_c} \psi \ind{_{a
            b}} \Biggr. & = \Biggl. - \frac{1}{2} \left( \omega
          \ind{_{c \lb{a}}} \nabla \ind{_{\rb{b}}} \sigma +\sum_d g \ind{_{c
              \lb{a}}} J \ind{_{\rb{b}} ^d} \nabla \ind{_d} \sigma
        \right) \Biggr) 
\end{align}
for all vector fields $\bm{X}$. Here, $\nabla$ is the Levi-Civita
covariant derivative, and $\bm{J}$ the complex structure. Contracting equation \eqref{Hamil2} with the
K\"{a}hler form $\bm{\omega}$ yields $\sigma = \tr _{\bm{\omega}}
\bm{\psi}$, the trace of $\bm{\psi}$ with respect to
$\bm{\omega}$.

In \cite{Martelli2005}, a new class of five-dimensional toric
Einstein-Sasaki manifolds is constructed by taking the BPS
(supersymmetric) limit of a four-dimensional black hole solution
similar to the Kerr-NUT-(A)dS metric. These limiting cases were later
generalised to arbitrary dimensions in
\cites{Chen2006,Hamamoto2007}. Broadly, under the change of
coordinates $x _\mu \rightarrow x _\mu= 1 + \varepsilon \xi_\mu$ and
after linear redefinitions of the coordinates $\psi_k$ and the
constants, in the limit $\varepsilon \rightarrow 0$, the
$2m$-dimensional Kerr-NUT-(A)dS metric becomes
\begin{align*}
  \bm{g} = \sum _{\mu=1}^{m} \left( \tilde{\bm{e}} ^\mu \odot \tilde{\bm{e}} ^\mu + \tilde{\bm{e}} ^{m+\mu} \odot \tilde{\bm{e}} ^{m+\mu} \right)
\end{align*}
where, in terms of the local coordinates $\left\{ \xi_\mu, t_k \right\}$,
\begin{align*}
  \tilde{\bm{e}} ^\mu & = \left(\frac{\Delta_\mu}{\Theta _\mu}\right) ^{1/2} \dd \xi_\mu, &
  \tilde{\bm{e}} ^{m + \mu} & = \left(\frac{\Theta_\mu}{\Delta _\mu}\right) ^{1/2} \sum_{k=1} ^{m} \sigma_\mu ^{(k)} \dd
  t _k,
\end{align*}
with
\begin{align*}
\Delta_\mu  & = \prod_{\substack{\nu=1\\\nu \neq \mu}} ^{m}
(\xi_\nu -
\xi_\mu), &
\sigma_\mu ^{(k)} & =
\sum_{\substack{\nu_1 < \nu_2 < \ldots < \nu_k \\ \nu_i \neq \mu}}  \xi_{\nu_1} \xi_{\nu_2}
\ldots \xi_{\nu_k}, & \sigma^{(k)} & =
\sum_{\nu_1 < \nu_2 < \ldots < \nu_k}  \xi_{\nu_1} \xi_{\nu_2}
\ldots \xi_{\nu_k},
\end{align*}
and where $\Theta _\mu = \Theta _\mu (\xi_\mu)$ are functions of one
variable. In this limiting case the almost-hermitian structure
$\bm{\omega} = \sum_\mu \bm{e} ^\mu \wedge \bm{e} ^{m+\mu}$ becomes 
\begin{align*}
        \bm{\omega} = \sum _{k=1}^m \dd \sigma^{(k)} \wedge \dd t_k,
\end{align*}
which is closed, and hence, $\bm{\omega}$ is K\"{a}hler.

The metric $\bm{g}$ is thus K\"{a}hler, and turns out to be
Ricci-flat. Further, as pointed out in
\cites{Martelli2005,Hamamoto2007}, it is identical to the
\emph{orthotoric} metric that has been found independently in
\cite{Apostolov2006}.  Such metrics are characterised by the existence of a Hamiltonian
  $2$-form.
The Hamiltonian $2$-form for the above metric is given
explicitly by 
\begin{align} \label{HamKNA}
        \bm{\psi} = \sum_\mu \xi_\mu \tilde{\bm{e}}^\mu \wedge \tilde{\bm{e}} ^{m+\mu}.
\end{align}
The striking parallel between $*$-Killing $2$-forms on Einstein manifolds and Hamiltonian $2$-forms on K\"{a}hler manifolds leads us to the natural question of whether the latter play a r\^{o}le similar to that of the former in the context of integrable isotropic distributions. The answer to this question is yes, and the result follows directly by a computation of the components of the connection $1$-form as in the proof of Theorem \ref{eventhm}.

\begin{thm} \label{Hamfol}
        Let $(M, \bm{g}, \bm{J}, \bm{\omega})$ be a $2m$-dimensional K\"{a}hler manifold equipped with a non-degenerate Hamiltonian $2$-form $\bm{\psi}$ with distinct eigenvalues. Then, the $2^m$ maximal isotropic distributions associated to $\bm{\psi}$ are integrable and define $2^m$ distinct complex structures.
\end{thm}

\begin{proof}
	The defining equation \eqref{Hamil2} for the Hamilton $2$-form gives in terms of the null basis
	\begin{subequations}
	\begin{align}
		\partial \ind{_\kappa} \psi \ind{_\mu ^\nu} + (\lambda_\mu - \lambda_\nu) \Gamma \ind{_\kappa _\mu ^\nu} & = - \frac{\ii}{2}\delta \ind*{_\kappa ^\nu} \partial \ind{_\mu} \sigma \label{Ham1a}\\
		(\lambda_\mu + \lambda_\nu) \Gamma \ind{^\kappa _\mu _\nu} & = 0 \label{Ham1b}\\
		(\lambda_\mu + \lambda_\nu) \Gamma \ind{_\kappa _\mu _\nu} & = 0 \label{Ham1c}
	\end{align}
	\end{subequations}
Equation \eqref{Ham1a} implies further
	\begin{subequations}
	\begin{align}
		\partial _\mu \lambda _\mu & = - \frac{\ii}{2} \partial _\mu \sigma \\
		\partial _\nu \lambda _\mu & = 0 & \mbox{for all $\nu \neq \mu$}\\
		(\lambda_\mu - \lambda_\nu) \Gamma \ind{_\nu _\mu ^\nu} & = - \frac{\ii}{2} \partial _\mu \sigma & \mbox{for all $\nu \neq \mu$}\\
		(\lambda_\mu - \lambda_\nu) \Gamma \ind{_\kappa _\mu ^\nu} & = 0 & \mbox{for all distinct $\kappa$, $\mu$, $\nu$}. \label{Cond+}
	\end{align}
	\end{subequations}
Since by assumption the eigenvalues are distinct, the integrability conditions \eqref{int2m} are satisfied.
\end{proof}

We also know \cites{Apostolov2006,Semmelmann2003} that given a Hamiltonian $2$-form $\bm{\psi}$, the $2$-form $\bm{\phi}$ defined by
\begin{align*}
        \bm{\phi} \equiv \bm{\psi} - \frac{1}{2} \sigma \bm{\omega}
\end{align*}
is a conformal Killing $2$-form. Such a $\bm{\phi}$ will not be closed in general. In fact,
\begin{align}
        \dd \bm{\phi} = - \frac{3}{n-1} \bm{\omega} \wedge \bm{J}(\dd^* \bm{\phi}), \label{clcocl}
\end{align}
so that $\bm{\phi}$ is closed iff it is co-closed iff it is parallel. On examining equation \eqref{LC+} and the eigenvalues of $\bm{\phi}$, one can see that equation \eqref{clcocl} is equivalent to the vanishing of the connection components $\Gamma \ind{^\nu _\mu _\nu}$ and $\Gamma \ind{_\nu ^\mu ^\nu}$ as implied by equation \eqref{Cond+}. This thus provides an example of a non-closed conformal Killing-Yano tensor which gives rise to $2^m$ integrable complex structures.

\section{Foliating spinors}
In four dimensions, maximal isotropic planes correspond to spinors up
to scale, \cite{Penrose1986}, and so spinors provide an efficient and
convenient calculus for studying such isotropic planes.  In higher
dimensions, spinors are less efficient as spin spaces grow in
dimension exponentially, and the condition that a spinor is `pure',
i.e., that it corresponds to a maximal isotropic plane becomes
non-trivial.  Nevertheless, they form a natural formalism for
understanding these structures.  In particular, the generalized Kerr
theorem \cites{Hughston1988} shows that maximally isotropic foliations
of complexified flat space-time are in $1:1$ correspondence with
holomorphic $m$-surfaces in twistor space and this can be identified
with the bundle of pure spinors over a Euclidean signature real slice.
In the following we give a spinor formulation of our previous results.
This will perhaps also be of benefit in answering other spinorial questions,
such as, for example, the separation of variables in spinor equations
in a space-time with a Killing-Yano tensor.

In this section, we first recall the basic facts of spin geometry. Our
exposition is based on various sources \cites{Gualtieri2003,
  Penrose1986, Lawson1989}, and we have harmonised the different
approaches as far as possible.  We then recast the previous results in
the language of spinors.

\subsection{The split signature model}
Let $V$ be an $m$-dimensional real vector space with dual
$V^*$, and consider the direct sum $V \oplus V^*$ endowed with an
inner product $\bm{g}$ of split signature $(m,m)$. One can
always find a null basis $\left\{ \bm{\theta}_a \right\} _{a=1}
^{2m} = \left\{ \bm{\theta}_{\mu}, 
\bm{\theta}^{\mu} \right\} _{\mu=1} ^m $ of $V \oplus V^*$, i.e.
\begin{align*}
  \bm{g} ( \bm{\theta}^{\mu}, \bm{\theta}_{\nu} ) = \delta _\nu
  ^\mu \qquad \mbox{and} \qquad \bm{g} ( \bm{\theta}^{\mu},
  \bm{\theta}^{\nu} ) = 0 = \bm{g} ( \bm{\theta}_{\mu},
  \bm{\theta}_{\nu}   ),
\end{align*}
for all $\mu, \nu$, so that the inner product is given by
\begin{align*}
  \bm{g} = \sum 2 \bm{\theta} \ind{^{\mu}} \odot \bm{\theta}
  \ind{_{\mu}}.
\end{align*}

The canonical basis for the exterior algebra $\bigwedge^\bullet (V
\oplus V^*)$ is induced from the basis of $V \oplus V^*$, and we will
often use the notation 
\begin{align*}
        \bm{\theta} \ind{^{\mu_1 \ldots \mu_p} _{\nu_1 \ldots \nu_q}}
        & \equiv \bm{\theta} \ind{^{\mu_1}} \wedge \ldots \wedge
        \bm{\theta} \ind{^{\mu_p}} \wedge \bm{\theta} \ind{_{\nu_1}}
        \wedge \ldots \wedge \bm{\theta} \ind{_{\nu_q}},\\ 
        \bm{\theta} \ind{^{a_1 \ldots a_{p+q}}} & \equiv \bm{\theta}
        \ind{^{a_1}} \wedge \ldots \wedge \bm{\theta}  \ind{^{a_{p+q}}}, 
\end{align*}
where $1 \leq \mu_i, \nu_i,  a_i  \leq m$. We denote by $\bm{1}$ the
basis element of $\bigwedge^0 (V \oplus V^*) \cong \mathbb{R} \cong
\bigwedge^0 V  \cong \bigwedge^0 V^*$. 

\begin{rem}
  Any even-dimensional real vector space with a positive definite
  metric $\bm g$, once complexified, admits a splitting $V \oplus \bar{V}$
  where the anti-holomorphic subspace $\bar{V}$ is isotropic and can
  be identified with the dual space $V^*$ via the Hermitian inner
  product induced by $\bm g$. If $\left\{ \bm{\theta}^a \right\} _{a=1} ^{2m} \equiv
  \left\{ \bm{\theta}^{\mu}, \bm{\bar{\theta}}^{\bar{\mu}}
  \right\}_{\mu, \bar{\mu}=1, \ldots, m}$ are the complex basis
  $1$-forms, then
\begin{align*}
  \bm{g}: \bm{\bar{\theta}} _{\bar{\mu}} \rightarrow \bm{\theta}
  ^\mu,
\end{align*}
where
\begin{align*}
  \bm{g} = \sum 2 \bm{\theta}^{\mu} \odot \bm{\bar{\theta}}
  ^{\bar{\mu}}.
\end{align*}
Similarly we can see that all of the subsequent results on $V \oplus
V^*$ apply to all signatures on the understanding that they will need
to be applied to the complexification of 
$V \oplus V^*$.
\end{rem}

A subspace $N$ of $V \oplus V^*$ such that $N \subseteq N^\perp$ is
called \emph{isotropic} and \emph{maximal isotropic} when strict
equality holds. Under the action of the Hodge duality operator, the space of all maximal isotropic subspaces splits
into \emph{self-dual} (SD) and \emph{anti-self-dual} (ASD) 
components. When $V \oplus V^*$ is complexified we have a one-to-one correspondence between SD (ASD) maximal isotropics and orthogonal complex structures with positive (negative) orientation.

\subsection{Spin representation}
The spin representation $\mathbb{S}$ of the special orthogonal group $\SO(V
\oplus V^*)$ is a $2^m$-dimensional vector space, which splits into two $2^{m-1}$-dimensional
irreducible representations $\mathbb{S}^+$ and $\mathbb{S}^-$. These are the chiral
spin representations of $\SO(V \oplus V^*)$. We shall give two alternative approaches to the theory of spinors, both of which will be used in the present paper according to the context.

The Clifford algebra can be regarded as a matrix algebra consisting of \emph{$\gamma$-matrices} satisfying the Clifford equation
\begin{align*}
        \gamma _a \gamma _b + \gamma _b \gamma _a = - 2 g _{a b} \mathbf{I},
\end{align*}
where $\mathbf{I}$ is the identity on $\mathbb{S}$.  Introduce a basis $\left\{ \bm{\theta}_\alpha \right\} = \left\{ \bm{\theta}_A, \bm{\theta}_{A'} \right\}$ of $\mathbb{S}=\mathbb{S}^+ \oplus \mathbb{S}^-$, so that lower-case Greek indices (beginning of the alphabet) running from $1$ to $2^m$ refer to  $\mathbb{S}$, and unprimed and primed upper-case Roman indices running from $1$ to $2^{m-1}$ to $\mathbb{S}^+$ and $\mathbb{S}^-$ respectively, with $\alpha = A \oplus A'$, and similarly for the dual spin spaces.
The action on each of the chiral spin spaces $\mathbb{S}^\pm$ can similarly be expressed in terms of `reduced' $\hat{\gamma}$- and $\check{\gamma}$-matrices
\begin{align*}
        \gamma \ind{_a _\alpha ^\beta} & =
                \begin{pmatrix}
                        0                       &
                        \hat{\gamma} \ind{_a _A ^{B'}} \\ 
                        \check{\gamma} \ind{_a _{A'} ^B}        &       0
                \end{pmatrix}
\end{align*}
satisfying the relations
\begin{align*}
        \hat{\gamma} \ind{_a} \check{\gamma} \ind{_b} + \hat{\gamma}
        \ind{_a} \check{\gamma} \ind{_b}  = -2 g \ind{_{a b}}
        \mathbf{I}^+ \qquad \mbox{and} \qquad         \check{\gamma}
        \ind{_a} \hat{\gamma} \ind{_b} + \check{\gamma} \ind{_b}
        \hat{\gamma} \ind{_a} = -2 g \ind{_{a b}} \mathbf{I}^-, 
\end{align*}
where $\mathbf{I}^\pm$ are the identity endomorphisms on
$\mathbb{S}^\pm$. In the following statements, $\hat{\gamma}$- and
$\check{\gamma}$-matrices may be substituted for $\gamma$-matrices in
an appropriate way.  Thus, one can express Clifford multiplication $\cdot$, i.e.
the action of the Clifford group on
$\mathbb{S}$, as follows: given a vector $\bm{V} = V ^a \bm{\theta}
_a$ and a spinor $\bm{\zeta}$, then 
\begin{align*}
\bm{V} \cdot \bm{\zeta} & = V ^a  \gamma_a \bm{\zeta}.
\end{align*}

On the other hand, we note that $\mathbb{S}$ is isomorphic as a vector
space to the exterior algebra $\bigwedge^\bullet V^*$. More precisely,
$\mathbb{S}^+$ and $\mathbb{S}^-$ are isomorphic to
$\bigwedge^{\mathrm{even}} V^*$ and 
$\bigwedge^{\mathrm{odd}} V^*$ with canonical bases $\left\{\bm{1},
  \bm{\theta} ^{\mu_1 \ldots \mu_p} : \mbox{$p$ even} \right\}$ and
$\left\{\bm{\theta} ^{\mu_1 \ldots \mu_p}: \mbox{$p$ odd} \right\}$
respectively. In this setting, Clifford multiplication is given
explicitly by\footnote{Our convention differs from
  \cite{Gualtieri2003} where the Clifford multiplication squares to
  \emph{plus} the norm squared.} 
\begin{align*}
  (\bm{X} + \bm{\xi}) \cdot \bm{\zeta} = - \bm{X} \lrcorner \bm{\zeta} +
  \bm{\xi} \wedge \bm{\zeta}
\end{align*}
for all vectors $\bm{X} + \bm{\xi} \in V \oplus V^*$.

The advantage of the former approach is conciseness of notation when dealing with purely spinorial quantities. However, the latter provides more practical  tools when it comes to computations in arbitrary dimensions. We shall set up a convenient dictionary between the two formalisms by identifying  each of the basis elements $\left\{ \bm{\theta}_\alpha \right\}$ with each of the basis elements $\left\{ \bm{1}, \bm{\theta} ^{\mu_1 \ldots \mu_p} \right\}$ of $\mathbb{S} \cong \bigwedge ^\bullet$ as follows.
\begin{align*}
        \bm{\theta} _\alpha & \leftrightarrow \bm{1}, \bm{\theta} ^{\mu_1 \ldots \mu_p} & \mbox{for any $1 \leq p \leq m$}\\
        \bm{\theta} _A & \leftrightarrow \bm{1}, \bm{\theta} ^{\mu_1 \ldots \mu_p}  & \mbox{for any even $1 \leq p \leq m$}\\
        \bm{\theta} _{A'} & \leftrightarrow \bm{\theta} ^{\mu_1 \ldots \mu_p} & \mbox{for any odd $1 \leq p \leq m$}.
\end{align*}
One may regard the indices $\alpha$, $A$, and $A'$ as labels for a group of indices $\mu_0 \ldots \mu_p$. Both types of bases will be regarded as the canonical bases of $\mathbb{S}$, $\mathbb{S}^+$ and $\mathbb{S}^-$ induced from the basis of $V$.

Since the Clifford algebra is isomorphic to the exterior algebra
$\bigwedge^\bullet (V \oplus V^*)$, one can extend the Clifford
multiplication to any elements of $\bigwedge ^\bullet (V \oplus V^*)$.
Writing $\gamma \ind{_{{a_1} \ldots {a_p}}} = \gamma \ind{_{\lb{a_1}}}
\ldots \gamma \ind{_{\rb{a_p}}}$, and employing the summation
convention from hereon until the end of the subsection, then, for any
$p$-form $\bm{\phi}= 
\phi _{a_1 \ldots a_p} \bm{\theta} ^{a_1 \ldots a_p}$ and any spinor
$\bm{\zeta}$, we have $\bm{\phi} \cdot \bm{\zeta} = \phi _{a_1 \ldots
  a_p} \gamma ^{a_p \ldots a_1} \bm{\zeta}$. Of particular importance
is the Lie algebra $\mathfrak{so}(V \oplus V^*)$, which is isomorphic
to $\bigwedge^2 (V \oplus V^*)$.  Any element $\bm{\phi}$ of
$\mathfrak{so}(V \oplus V^*)$ admits the 
decomposition
\begin{align}
  \bm{\phi} =
  \begin{pmatrix}
    \bm{A} & \bm{\beta} \\
    \bm{B} & -\bm{A}^*
  \end{pmatrix},
\end{align}
where $\bm{A} \in \End{V}$, and $\bm{\beta} \in \Hom(V^*,V)$ and
$\bm{B} \in \Hom(V,V^*)$ are skew. The action of $\bm{\phi}$ on
spin space $\mathbb{S}$ is then given by
\begin{align} \label{Clifact}
  \bm{B} \cdot \bm{\zeta} & = B _{\mu \nu} \bm{\theta}^\nu \wedge (\bm{\theta}^\mu \wedge \bm{\zeta}) =
- \bm{B} \wedge \bm{\zeta} & (B _{\mu \nu} & = B _{[\mu \nu]})\nonumber \\
  \bm{\beta} \cdot \bm{\zeta} & = \bm{\beta}^{\mu \nu} \bm{\theta} _\nu \lrcorner
  (\bm{\theta} _\mu
\lrcorner \bm{\zeta}) = \bm{\beta} \lrcorner \bm{\zeta} & (\beta ^{\mu \nu} & = \beta ^{[\mu \nu]})\\
  \bm{A} \cdot \bm{\zeta} & = A \ind{^\nu _\mu}
  \bm{\theta} ^\mu
\wedge (\bm{\theta} _\nu \lrcorner \bm{\zeta}) - \frac{1}{2} \tr \bm{A} \bm{\zeta} = \bm{A}^*
\bm{\zeta} - \frac{1}{2} \tr \bm{A} \bm{\zeta} & (A \ind{^\nu _\mu} & = - A \ind{_\mu ^\nu}) \nonumber,
\end{align}
for any spinor $\bm{\zeta}$.

Spin space $\mathbb{S}$ is equipped with an inner product $<\cdot,\cdot>$ which is symmetric or anti-symmetric according to $m$. This inner product descends to an inner product on each of the chiral spin spaces $\mathbb{S}^\pm$ when $m$ is even, but it is degenerate on $\mathbb{S}^\pm$ when $m $ is odd, in which case it gives rise to an isomorphism between a space of one chirality and the dual of the space of the opposite chirality, i.e. $\mathbb{S}^{\pm} \cong \left( \mathbb{S}^{\mp} \right)^*$. In general, given any two spinors $\bm{\eta}$ and $\bm{\zeta}$ one can define a $p$-form $\bm{\phi}$ by
\begin{align*}
        \bm{\phi} & = < \bm{\eta},  \bm{\theta} \ind{^{a_1 \ldots a_p}} \gamma \ind{_{a_1 \ldots a_p}} \bm{\zeta} >.
\end{align*}

\subsection{Maximal isotropic planes and pure spinors}
To any non-zero spinor $\bm{\zeta}$ one can associate  an isotropic
subspace given by
\begin{align*}
  N ({\bm{\zeta}}) = \left\{ \bm{X}+\bm{\xi} \in V \oplus V^* : (\bm{X} + \bm{\xi}) \cdot \bm{\zeta} = 0
  \right\},
\end{align*}
and any element $ \bm{X}+\bm{\xi} \in N ({\bm{\zeta}})$ has the form
\begin{align*}
         \bm{X}+\bm{\xi} = < \bm{\eta},  \bm{\theta} _a \gamma ^a \bm{\zeta} >
\end{align*}
for some spinor $\bm{\eta}$. If $N ({\bm{\zeta}})$ is maximal we say that $\bm{\zeta}$ is a
\emph{pure} spinor. In particular, the $2^m$ basis elements $\left\{ \bm{1}, \bm{\theta} ^{\mu_1 \ldots \mu_p} \right\}$ of $\mathbb{S}$ are pure with associated maximal
isotropic planes
\begin{align*}
N ({\bm{1}}) & =
 \left\{ \bm{\theta} _{1}, \ldots,  \bm{\theta} _{m} \right\} \oplus \left\{ \bm{0} \right\},\\
  N ({\bm{\theta} ^{\mu_1 \ldots \mu_p}}) & =
 \left\{  \bm{\theta} _{\mu_{p+1}}, \ldots, \bm{\theta} _{\mu_{m}}
  \right\} \oplus
 \left\{ \bm{\theta} ^{\mu_1}, \ldots,  \bm{\theta} ^{\mu_p} \right\},\\
N ({\bm{\theta} ^{1 \ldots m}}) & =
        \left\{ \bm{0} \right\} \oplus \left\{ \bm{\theta} ^{1}, \ldots,  \bm{\theta} ^{m} \right\},
\end{align*}
where $\mu_i \neq \mu_j$ for all $i \neq j$.

One can show that a spinor is pure if and only if it is chiral. We
denote the space of $\pm$ chiral pure spinors by $\mathbb T^\pm$. The
projective pure spinors $\mathbb{PT}^\pm$ are the spaces $\mathbb
T^\pm$ defined up to scalings.  There is a one-to-one correspondence
between projective pure spinors of a given chirality and maximal
isotropic planes of a given duality. The next proposition is a direct
consequence of equation \eqref{Clifact}.
\begin{prop} \label{espinpure}
  Let $\bm{\phi}$ be an element of $\mathfrak{so}(V \oplus V^*)$ such that $\bm{\phi}$ is diagonal in the null basis $\left\{ \bm{\theta}^\mu,
\bm{\theta}_\mu \right\}$, i.e.
\begin{align} \label{stKdiag}
\bm{\phi} = \sum_\mu \lambda _\mu \bm{\theta} ^{\mu} \wedge
\bm{\theta} _{\mu},
\end{align}
for some $\lambda _\mu$.
Then the eigenspinors of $\bm{\phi}$ regarded as an endomorphism on
$\mathbb{S}$ are simply the basis elements of $\bigwedge V^* (= \mathbb{S})$, i.e.
\begin{align*}
        \bm{\phi} \cdot \bm{1} & = \tilde{\lambda} _{0} \bm{1}, &
  \bm{\phi} \cdot \bm{\theta} ^{\mu_1 \ldots \mu_p} & = \tilde{\lambda} _{\mu_1 \ldots \mu_p} \bm{\theta} ^{\mu_1 \ldots \mu_p}
\end{align*}
where the eigenvalues $\tilde{\lambda} _0$, $\tilde{\lambda} _{\mu_1 \ldots \mu_p}$ are given in terms of the eigenvalues
$\lambda_\mu$ by
\begin{align*}
   \tilde{\lambda}_{0} & = - \frac{1}{2} \sum_\mu \lambda_\mu,
& \tilde{\lambda} _{\mu_1 \ldots \mu_p} & = - \frac{1}{2} \sum_\mu
(-1)^\epsilon \lambda_\mu
\end{align*}
where $\epsilon = \sum_{i=1}^p\delta \ind*{_\mu ^{\mu_i}}$. It follows that the
eigenspinors of $\bm{\phi}$ are pure. 
\end{prop}

\subsection{Twistor bundle and integrability condition}
Let $M$ be a real $2m$-dimensional (pseudo-) Riemannian spin manifold
so that at each point 
$p$, its complexified tangent space $\mathbb{C} \otimes T  _p M$ can be
given the structure of $\mathbb{C} \otimes (V \oplus V^*)$. The
preceding sections translate into the language of bundles in the
obvious way so that $V$, $\mathbb{S}$, $\mathbb{T}$, etc... will now
refer to bundles over the complexification $M_{\mathbb{C}}$ of $M$. We
extend the Levi-Civita covariant derivative $\nabla$ to a covariant
derivative on 
the spin bundle $\mathbb{S}$ -- also denoted $\nabla$. For any basis
spinor field $\bm{\theta}$ we have
\begin{align} \label{spinconn}
  \nabla \bm{\theta} = - \frac{1}{2} \sum \bm{\Gamma} _{a b} \bm{\theta} ^a \wedge \bm{\theta} ^b \cdot
  \bm{\theta},
\end{align}
where $\bm{\Gamma} \ind{_a ^b}$ is the Levi-Civita connection $1$-form on $T^* M$, and $\bm{\Gamma} _{a b} = \bm{\Gamma} _{[a b]}$.

Maximal isotropic distributions of $T M$
are in one-to-one correspondence with orthogonal almost complex
structures on $\mathbb{C} \otimes T M$ and sections of the projective twistor bundle
$\mathbb{PT}$ over $M_{\mathbb{C}}$  (i.e. projective pure spinor fields on $M$).
The Frobenius integrability condition can then be articulated as
follows.

\begin{prop}
 A maximal isotropic distribution or its
associated orthogonal almost complex structure is integrable if and only if the
associated projective pure spinor field $\bm{\zeta}$ satisfies
\begin{align} \label{prspintcond}
  \left( \nabla _{\bm{X}} \bm{\zeta} \right) \wedge \bm{\zeta} = 0, \quad \mbox{i.e.} \quad \nabla _{\bm{X}} \bm{\zeta} = f _{\bm{X}} \bm{\zeta},
\end{align}
for all vector fields $\bm{X} \in \Gamma (N_{\bm{\zeta}})$ and for some
function $f$ on the manifold depending on ${\bm{X}}$. In other words, integrable orthogonal almost complex structures correspond to
holomorphic sections of the projective twistor bundle. 
\end{prop}
If one applies equation \eqref{prspintcond} to all the (projective) basis
elements of the spin bundle and adopts the following convention
\begin{align*}
\Gamma \ind{_{c a} ^b} = \bm{\theta} \ind{_c} \lrcorner
\bm{\Gamma} \ind{_a ^b}
\end{align*}
for the
components of $\bm{\Gamma} \ind{_a ^b}$, one then obtains
\begin{prop} \label{intspinbasis}
  All $2^m$ projective basis elements $\bm{1}$, $\bm{\theta} ^{\mu_1 \ldots \mu_p}$ of the projective twistor bundle  $\mathbb{PT} \subset \mathbb{PS} \cong \mathbb{P} (\bigwedge ^\bullet T V^*)$ are integrable if and only if the connection components
\begin{align} \label{connintcond}
  & \Gamma \ind{_{\kappa \mu \nu}}, \,
  \Gamma \ind{^{\kappa \mu \nu}}, \qquad (\mbox{for all
  $\mu,\nu,\kappa$}), \nonumber \\
  & \Gamma \ind{_\kappa ^{\mu \nu}}, \,
  \Gamma \ind{^\kappa _{\mu \nu}}, \qquad (\mbox{for all $\kappa \neq
  \mu,\nu$}), \\
  & \Gamma \ind{_{\kappa \mu} ^\nu} , \, \Gamma \ind{^\kappa _\nu ^\mu}, \qquad (\mbox{for all $\nu
  \neq \kappa, \mu$}), \nonumber
\end{align}
all vanish.
\end{prop}

\begin{rem}
        We note that  equations \eqref{connintcond} are equivalent to equations \eqref{Riccirot} obtained by the Frobenius integrability condition.
\end{rem}

By Proposition \ref{espinpure}, we have
\begin{cor} \label{espint}
  The eigenspinors of any spin endomorphism on $\mathbb{S}$ of the
  form \eqref{stKdiag} are integrable if and only if the components of
  the connection $1$-form \eqref{connintcond} all vanish. 
\end{cor}

We can then reformulate Theorem \ref{eventhm} as follows.
\begin{thm} \label{eventhmspin}
Let $M$ be a $2m$-dimensional spin manifold equipped with a conformal Killing-Yano tensor $\bm{\phi}$ as in Theorem \ref{eventhm}. Then the $2^m$ eigenspinors of $\bm{\phi}$ are integrable.
\end{thm}

\subsection{Weyl curvature restrictions revisited}

We can also give a spinorial articulation of Theorem \ref{TWeyl} in the same vein as \cite{Penrose1986}.
Denote by $\bm{\Psi}$ and by $\bm{\Psi}^{\pm}$ the completely traceless elements of $\bigodot ^2 \End(\mathbb{S})$ and $\bigodot ^2 \End(\mathbb{S}^{\pm})$ corresponding to the Weyl tensor $\bm{C}$ viewed as an element of $\bigodot ^2 \mathfrak{so}(n)$. In spin components,
\begin{align*}
        {\Psi^+} \ind*{_{A C} ^{B D}} \equiv \hat{\gamma} \ind{^{a} _A ^{E'}} \check{\gamma} \ind{^{b} _{E'} ^B} C \ind{_{a b c d}} \hat{\gamma} \ind{^{c} _C ^{F'}} \check{\gamma} \ind{^{d} _{F'} ^D} \quad \mbox{and} \quad
        {\Psi^-} \ind*{_{A' C'} ^{B' D'}} \equiv \check{\gamma} \ind{^{a} _{A'} ^{E}} \hat{\gamma} \ind{^{b} _{E} ^{B'}} C \ind{_{a b c d}} \check{\gamma} \ind{^{c} _{C'} ^{F}} \hat{\gamma} \ind{^{d} _{F} ^{D'}},
\end{align*}
where we have used the summation convention.
Theorem \ref{TWeyl} can then be reformulated in spinorial terms:
\begin{thm} \label{WeylspinDTh}
        Let $\bm{\phi}$ be a non-degenerate conformal Killing-Yano tensor
        with distinct eigenvalues diagonal in the null basis
        $\left\{ \bm{\theta} _a \right\}$. Then each of the
        eigenspinors $\bm{\theta}_{A}$, $\bm{\theta}_{A'}$ ($A, A'=1,
        \ldots, 2^{m-1}$) of the corresponding spin endomorphisms
        satisfies (no summation)
\begin{align}  \label{WeylspinD}
        \bm{\Psi}^+ \left( \bm{\theta}_A, \bm{\theta}_A \right) \wedge
        \bm{\theta}_A = 0 \qquad \mbox{and} \qquad \bm{\Psi}^- \left(
          \bm{\theta}_{A'}, \bm{\theta}_{A'} \right) \wedge
        \bm{\theta}_{A'} = 0. 
\end{align}
\end{thm}

In four dimensions, each spin space $\mathbb{S}^\pm$ is a
two-dimensional complex vector space equipped with a symplectic inner
product, so that any symmetric spinor of valence $p$ is fully
decomposable as a symmetric product of $p$ spinors of valence $1$. In
particular, the $\bm{\Psi}^\pm$ admit such a decomposition, and
equations \eqref{WeylspinD} are equivalent to
%\footnote{In abstract
%  index notation, this expression is equivalent to ${\Psi^+} \ind{_{A
%      B C D}} = \Psi_2 \theta \ind*{^1_{\lp{A}}} \theta  \ind*{^1_B}
%  \theta  \ind*{^2_C} \theta \ind*{^2_{\rp{D}}}$.} 
\begin{align*}
        {\bm{\Psi}^+} = \Psi_2 \bm{\theta} ^1 \odot \bm{\theta} ^1
        \odot \bm{\theta}  ^2 \odot \bm{\theta} ^2, 
\end{align*}
where $\Psi_2 \equiv \Psi_{1122}$, and similarly for $\bm{\Psi}^-$.
By the Goldberg-Sachs theorem, it follows that each of the spinors
$\bm{\theta}_A$, $\bm{\theta}_{A'}$, ($A, A' =1,2$), satisfies the
integrability condition \eqref{prspintcond}.

\begin{rem}
A possible
classification of  $\bm{\Psi}^\pm$ in six dimensions and its relation
with integrable spinors were first investigated in
\cites{Jeffryes1995,Hughston1995,Mason1995}. 
\end{rem}

\subsection{Spin bundle over odd-dimensional manifolds}
Let $M$ be a $(2m+1)$-dimensional Riemannian spin manifold. Then, the complexification $T_\mathbb{C} M$ of the tangent bundle admits a splitting $V \oplus V^* \oplus K$, where $V$ and $V^*$ are $m$-dimensional vector bundles dual (and conjugate) to each other, and $K$ is a complex line bundle. The spin bundle $\mathbb{S}$ over $M$ is now
irreducible and isomorphic to $\bigwedge V^*$. Maximal isotropic distributions of $V \oplus V^*$ correspond
to sections of the projective twistor (or pure spinor) bundle $\mathbb{PT}$ over
$M$, or equivalently, to orthogonal almost complex structures on $V
\oplus V^*$. When such an isotropic distribution is integrable, the
corresponding section of $\mathbb{PT}$ is holomorphic, and the corresponding almost
CR structure of $M$ integrable. The eigenspinors of a conformal Killing-Yano tensor $\bm{\phi}$ are precisely the basis
elements of the spin bundle $\mathbb{S}$ induced from the basis of
$V^*$. Now, assuming that $\bm{\phi}$ has distinct eigenvalues, Theorem \ref{eventhmspin} extends naturally to odd-dimensional manifolds.

\section{Concluding remarks and applications}

\subsection{Intersection of foliations and reality conditions}
To obtain similar results on a real pseudo-Riemannian manifold $M$, it
suffices to impose suitable reality conditions on the complexified
tangent bundle $T_\mathbb{C} M$ . Given a real manifold
equipped with a pseudo-Riemannian metric of signature $s$, the
intersection of an integrable maximal isotropic distribution $\D$ and
its (integrable) conjugate $\bar{\D}$ gives rise \cite{Hughston1988}
to an integrable real isotropic distribution $ K=\D\cap\bar\D$ whose
rank can be any of $(2m-|s|)/2$ modulo 2 (in
\cite{Hughston1988} it was claimed in error that the rank is always
$(2m-|s|)/2$).  As we have shown earlier, in the complexification, the
integral surfaces of $\D$ are totally geodesic, and so therefore are
those of $K$.  In the case of Lorentzian manifolds, where $|s|=2m-2$,
the foliation $K$ is 1-dimensional and tangent to a congruence of null
geodesics and the \emph{screen space} $K^\perp/K$ of $K$ inherits the
complex structure on $T_\mathbb{C} M$ from $\D$ that is Lie
derived along the congruence 
\cite{Hughston1988,Nurowski2002}. In four dimensions, the preservation
fo the complex structure on the screen space is equivalent to 
the \emph{shear-free} condition,
i.e., the preservation of the \emph{conformal} structure of $K^\perp/K$ along
$K$. This is a consequence of the fact that complex structures
and conformal metrics on a surface are the same.
However, this is not true in higher dimensions, and a six-dimensional
counter-example invalidating this equivalence between being shear-free
and preserving a complex structure in higher dimensions is given in
\cite{Trautman2002a}.

Spinorially, the real structure of $M$ induces a complex conjugation
$\mathcal{C}$ on the spin bundle, which preserves each of the chiral
spin bundles when $s/2$ is even, and interchanges them when $s/2$ is
odd. Depending on $s$, $\mathcal{C}$ may be quaternionic
\cite{Penrose1986}, i.e. $\mathcal{C}^2=-1$. In the Lorentzian case, a
real vector field $\bm{k}$ as a section of the isotropic line bundle
$K$ as defined above can then be expressed as\footnote{This is
  \emph{always} possible and guaranteed by the various properties of
  the spin inner product and the conjugation in different dimensions
  and signatures.} 
\begin{align} \label{PND}
 \bm{k}= < \bar{\bm{\zeta}}, \bm{e}_a \gamma ^a \bm{\zeta} >,
\end{align}
where $\bm{\zeta}$ is an integrable pure spinor, and
$\bar{\bm{\zeta}}$ its conjugate under $\mathcal{C}$. In dimensions greater than six, the choice of the spinor $\bm{\zeta}$  for $\bm{k}$ is no longer unique.
A more general treatment of real structures and spinors is given in
\cite{Kopczy'nski1992}.  

\subsection{The Kerr-Schild ansatz}
The Kerr-NUT-(A)dS metric given in \cite{Chen2006} has a long
history that originates in the four-dimensional Kerr-Schild
ansatz found in \cite{Kerr1963} in 1963. The original aim of Kerr's
paper was to construct a solution to Einstein's equations which is of
Petrov type D with a twisting (i.e. hypersurface orthogonal) GPND
$\bm{k}$. It turns out that the newly-found metric is an \emph{exact}
first-order perturbation of the flat metric, and describes a rotating
black hole of mass $M$: 
\begin{align} \label{KerrSchild}
\bm{g} = \tilde{\bm{g}} + \frac{2 M}{U} \left( \bm{k} \right)^2.
\end{align}
where $\tilde{\bm{g}}$ is the Minkowski background metric, $U$ some
function, and $\bm{k}$ is a shear-free isotropic geodesic (real)
vector field with respect to both $\bm{g}$ and $\tilde{\bm{g}}$. Since
$\bm{k}$ is shear-free, it belongs to a maximal isotropic integrable
distribution of complexified Minkowski spacetime, and thus has the
form \eqref{PND}. 

The Kerr-Schild ansatz \eqref{KerrSchild} has been generalised in
higher dimensions in Lorentzian signature $(1,2m-1)$ in
\cite{Gibbons2005}. The background metric $\tilde{\bm{g}}$ is now
allowed to be a pure (A)dS $2m$-dimensional metric. The vector field
$\bm{k}$ is again isotropic and geodesic with respect to both $\bm{g}$
and $\tilde{\bm{g}}$. However, as shown in \cite{Pravda2004}, it fails
to retain its shear-free property in dimensions greater than four.
There does, however, remain the question of whether $\bm{k}$ arises
from one of our integrable maximal isotropic distributions (i.e., pure
spinors). For the Kerr-NUT-(A)dS
metric one can follow the various coordinates transformations leading from
the Kerr-Schild ansatz to the Kerr-NUT-(A)dS metric in
\cites{Gibbons2005,Chen2006}, and we find that the \emph{real} vector
$\bm{k}$ of the former is the same (up to factor) as one of the
\emph{complex} basis vectors $\left\{\bm{\theta}_\mu, \bm{\theta}^\mu
\right\}$ of the latter. By Theorem \ref{eventhm}, the eigenspinors of
the closed conformal Killing-Yano tensor \eqref{TwKNA} on the
Kerr-NUT-(A)dS metric are integrable, so that each $\bm{\theta}_{\mu}$
(and each $\bm{\theta}^{\mu}$) will belong to (an intersection of)
some integrable maximal isotropic distributions, and we can write
\begin{align*}
 \bm{\theta}_\mu= < \bm{\eta}, \bm{\theta}_a \gamma ^a \bm{\zeta} >,
\end{align*}
for some integrable (pure) basis spinors $\bm{\zeta}$ and $\bm{\eta}$. Again, there is some freedom in the choice of spinors.
Since there is no fundamental difference between the complexified Kerr-Schild metric and Kerr-NUT-(A)dS metric, it follows that by imposing a suitable reality condition, $\bm{k}$ will have the form \eqref{PND} in all (even) dimensions.

\begin{rem}
It is shown  in \cite{Chen2008} how one can obtain an $m$-Kerr-Schild ansatz in split signature $(m,m)$ from the Kerr-NUT-(A)dS metric by Wick rotating the coordinates $\psi_k$. After some work, we have
\begin{align*}
        \bm{g} = \tilde{\bm{g}} - \sum_{\mu=1} ^{m} \frac{2 b_\mu x_\mu}{U_\mu} \left( \bm{k}_{(\mu)} \right)^2,
\end{align*}
where $\tilde{\bm{g}}$ is the background pure (A)dS metric, the $b_\mu$ are the mass and the NUT parameters, and the functions $U_\mu$ and coordinates $x_\mu$ are as for the Kerr-NUT-(A)dS metric. The $m$ real vectors $\bm{k}_{(\mu)}$ are linearly independent, mutually orthogonal, isotropic and geodesic with respect to both $\bm{g}$ and $\tilde{\bm{g}}$. Thus, they span a \emph{real} maximal isotropic distribution. In fact, all $2^m$ maximal isotopic distributions arising from the closed conformal Killing-Yano tensors are real. The spin bundles over $M$ are also real, and, using the same argument as above, each $\bm{k}_{(\mu)}$ can be expressed as
\begin{align*}
 \bm{k}_{(\mu)}= < \bm{\eta}_{(\mu)}, \bm{e}_a \gamma ^a \bm{\zeta}_{(\mu)} >,
\end{align*}
for some appropriate choice of spinors $\bm{\zeta}_{(\mu)}$ and $\bm{\eta}_{(\mu)}$.
\end{rem}

The odd-dimensional versions of the above metrics are similar and share the same properties as their even-dimensional counterparts.

\subsection{The Kerr Theorem}
The Kerr Theorem provides a systematic method of finding shear-free null
geodesic vector fields arising from an integrable almost complex
structure in complexified Minkowski spacetime.  
The original theorem consists in solving $F=0$ where $F$ is a certain
holomorphic function of the \emph{complexified} isotropic flat
coordinates \cites{Kerr1963,Cox1976}. Penrose gave
the Kerr theorem a new and more geometric formulation by realizing $F$
as a function on twistor space in his original paper on twistor
geometry \cite{Penrose1967}. A 
generalisation to higher dimensions was given in \cite{Hughston1988}.
Essentially, it states that a pure spinor field on a
(complexified) flat $2m$-dimensional manifold $M$ equipped with a
conformal metric is integrable if and only if it can be determined by
the intersection of an $m$-dimensional analytic surface and the set of
projective pure spinor spaces in twistor space representing a region of
$M$. This surface is defined by $m(m-1)/2$ homogeneous holomorphic
functions on twistor space. In the context of the four-dimensional
type D, e.g., the Lorentzian Kerr-NUT metric, the integrable spinor field is
determined by a single quadratic function constructed from the angular
momentum twistor. In higher dimensions, the co-dimension and hence the
number of such functions
increases quadratically with the dimension (being the dimension of the
space of pure spinors), and a characterization of the structure has
yet to emerge.

\subsection{Degenerate conformal Killing $2$-forms and conformal Killing spinors}
As pointed out above, a degenerate conformal Killing-Yano tensor on a
four-dimensional Lorentzian manifold implies that the Weyl tensor is
of Petrov type N. By the Goldberg-Sachs theorem, this spinor is
integrable. In fact, on an $n$-dimensional Riemannian manifold, such a conformal Killing-Yano tensor arises as
the `squaring' of a \emph{conformal Killing spinor} or \emph{twistor
  spinor}, i.e. a spinor $\bm{\zeta}$ which satisfies the twistor
equation  
\begin{align*}
        \nabla_{\bm{X}} \bm{\zeta}+ \frac{1}{n} \bm{X} \cdot \Dop \bm{\zeta}& = 0,
\end{align*}
for all vector fields $\bm{X}$, where $\Dop$ is the Dirac operator \cite{Baum2000,Semmelmann2003}. When $\bm{\zeta}$ satisfies the Dirac equation $\Dop \bm{\zeta} = \lambda \bm{\zeta}$ for some function $\lambda$, $\bm{\zeta}$ is called a \emph{Killing spinor}. In the special case where $\lambda \equiv 0$, $\bm{\zeta}$ is a \emph{parallel spinor}. Clearly, a pure Killing spinor automatically satifies the integrability condition \eqref{prspintcond}. A conformal Killing spinor must also satisfy the integrablity condition
\begin{align} \label{typeN}
\bm{C}(\bm{X},\bm{Y}) \cdot \bm{\zeta} & = 0,
\end{align}
for all vector fields $\bm{X}$, $\bm{Y}$.

One can show \cite{Semmelmann2003,Cariglia2004} that given two conformal Killing spinors $\bm{\zeta}$ and $\bm{\eta}$, the $2$-form defined by
\begin{align*}
        \bm{\phi} = <\bm{\eta}, \bm{\theta} \ind{^{ab}} \gamma \ind{_{ab}} \bm{\zeta} >
\end{align*}
is a conformal Killing-Yano tensor\footnote{On Lorentzian four-dimensional spacetimes, equation \eqref{typeN} is the Petrov type N condition, and the conformal Killing-Yano tensor is obtained by taking $\bm{\zeta}$ pure together with its complex conjugate $\bm{\eta}=\bar{\bm{\zeta}}$, or equivalently a real (`Majorama') Dirac spinor. These are pp-wave spacetimes.}. Conformal Killing spinors have been extensively studied, and we refer the reader to the literature (e.g. \cite{Baum2000} and references therein) for details.

\bibliographystyle{plain}
\bibliography{CKYArxivSubm2}

\end{document}